\documentclass[12pt]{amsart}

\usepackage{amsmath}

\usepackage{amsthm}
\usepackage{graphicx}
\usepackage{lscape}
\usepackage[all]{xy}

\usepackage{epsfig}

\usepackage{amssymb}
\usepackage{latexsym}

\newtheorem{Theorem}{Theorem}
\newtheorem{Proposition}{Proposition}
\newtheorem{Lemma}{Lemma}
\newtheorem{Corollary}{Corollary}

\newtheorem{Claim}{Claim}

\theoremstyle{definition}

\newtheorem{remark}{Remark}
\newtheorem{Example}{Example}

\def\OO{{\mathcal O}}

\def\SO4{{\mathrm{SO(4)}}}
\def\O4{{\mathrm{O(4)}}}
\def\mod{{\mathrm{mod}}}

\def\C{{\mathbb C}}
\def\R{{\mathbb R}}
\def\H{{\mathbb H}}
\def\Z{{\mathbb Z}}

\begin{document}

\title{Fibered spherical 3-orbifolds}

\author[M. Mecchia]{Mattia Mecchia*}
\address{M. Mecchia: Dipartimento Di Matematica e Geoscienze, Universit\`{a} degli Studi di Trieste, Via Valerio 12/1, 34127, Trieste, Italy.} \email{mmecchia@units.it}
\author[A. Seppi]{Andrea Seppi**}
\address{A. Seppi: Dipartimento di Matematica ``Felice Casorati", Universit\`{a} degli Studi di Pavia, Via Ferrata 1, 27100, Pavia, Italy.} \email{andrea.seppi01@ateneopv.it}



\thanks{$^{*}$Partially supported by the FRA 2011 grant  ``Geometria e topologia delle variet\`{a}'', Universit\`{a} degli Studi di Trieste, and by the  PRIN 2010-2011 grant ``Variet\`{a} reali e complesse: geometria, topologia e analisi armonica''.}

\thanks{$^{**}$Partially supported by the FIRB 2011-2014 grant ``Geometry and topology of low-dimensional manifolds''.}

\begin{abstract}
In early 1930s Seifert and Threlfall classified up to conjugacy  the finite subgroups of $\SO4$, which gives an algebraic classification of orientable spherical 3-orbifolds.
For the most part, spherical 3-orbifolds are  Seifert fibered. The  underlying topological space and singular set of non-fibered spherical 3-orbifolds were described by Dunbar. In this paper we deal with the fibered case and in particular we give  explicit  formulae  relating the  finite subgroups of $\SO4$ with the invariants of the corresponding fibered 3-orbifolds. This allows us to deduce directly from the algebraic classification topological properties of spherical 3-orbifolds.   
\end{abstract}

%
%

\maketitle

\section{Introduction}

Geometric 3-manifolds and 3-orbifolds play an important role in the  geometrization program of Thurston (completed  at the beginning of this century by Perelman).

Roughly speaking, a n-orbifold is a Hausdorff topological space locally modelled by quotients of $\mathbb{R}^n$ by finite groups of isometries.  
To each point $x$ of the orbifold is associated the minimal group $\Gamma_x$ such that $x$ has a neighborhood modelled on $\R^n/\Gamma_x.$
If $\Gamma_x$ is non-trivial the point is called singular. Complete geometric orbifolds are orbifolds diffeomorphic to   the quotient of a geometric space  (e.g. spherical, Euclidean and hyperbolic space) by a discrete groups of isometries. In particular an orientable  spherical 3-orbifold is a  quotient of $S^3$ by a finite subgroup of $\SO4$. For basic definitions about orbifolds see for example \cite{BMP}.

In early 1930s Seifert and Threlfall classified up to conjugacy  the finite subgroups of $\SO4$ (\cite{TSe1} and \cite{TSe2}). 
The standard reference in English has been  the book of Du Val \cite{DV}, where the groups are divided in families and enumerated.
The classification of finite subgroups of $\SO4$ gives  immediately an algebraic classification of spherical 3-orbifolds, but 
from a topological point of view this classification is not completely satisfactory because it does not give any direct information about the topological structure of the orbifold (underlying topological space and singular set). 

W.D. Dunbar wrote two classical papers about geometric 3-orbifolds. In  \cite{Dun2} he classified the Seifert fibered geometric 3-orbifolds with underlying topological space $S^3$ in terms of the invariants of the fibration and  the singular set of these orbifolds was explicitly drawn. In  \cite{Dun3},  he described the topology of a non-fibered spherical 3-orbifold starting from the corresponding finite subgroup of $\SO4$. Up to conjugacy the groups giving a non-fibered 3-orbifold are 18, the remaining groups (that are collected in 24 infinite families) leave invariant a fibration of $S^3$ which induces a fibration on the 3-orbifold. In particular all  these groups preserve, up to conjugacy, the Hopf fibration of $S^3$.

In \cite{Dun3}  Dunbar wrote about fibered orbifolds: ``...these orbifolds are ame\-nable to study {\it en masse}, although in practice there are so many cases that it is hard to give a formula that will translate a description of a subgroup of $\SO4$ into a description of the corresponding orbifold (or vice versa)."

In this paper we do exactly this, for each family of groups preserving a fibration  we describe an explicit formula giving the invariants of the fibered quotient 3-orbifold. The fibration is that induced by the Hopf fibration. This permits  to deduce directly  from the group presentation in \cite{DV} some   topological information about the quotient orbifold. The results are collected in Tables~\ref{pluto},  \ref{pippo} and \ref{topolino}.

In Section~\ref{algebraic-classification} we review briefly the main ideas of the algebraic classification of the finite subgroups of $\SO4$. It is interesting to point out that Du Val's list is not complete and three new families of groups have to be inserted. 
In their book about quaternions and octonions \cite{CS}, J.H.Conway and D.A.Smith  have revisited  the classification,  giving a    complete list of the finite subgroups of $\SO4$. They do not use the same   notation of  Du Val; for details they often refer  to  Goursat's paper \cite{G}. 
Here we choose to use the same notation and enumeration as in Du Val's book, inserting the new three families of groups. In particular we describe  how the new families appear in the Du Val procedure of classification. 

In Section~\ref{Seifert orbifold} we recall briefly  the basic notions about Seifert fibered 3-orbifolds and analyze which groups in Du Val's list leave invariant a fibration of $S^3$.

Section~\ref{abelian}  is mainly dedicated to the analysis of the abelian groups (Families 1 and $1'$). From an algebraic point of view this seems the simplest case but,  using the presentation of the groups used in the classification,  the geometry of the action  appears quite obscure. For example it is not easy to understand which transformations have non-empty fixed point set. In fact to compute  the index of the singular points in the quotients we have to distinguish many cases  according to the parity of the four indices involved. 
For the abelian groups we compute explicitly the underlying topological space that is in every case a lens space. The singular set is contained in the union of the cores of the tori giving the lens space; the singularity indices of the two components  is explicitly given.  
The generalized dihedral case (Families 11 and $11'$) is a direct consequence of the abelian one. Here the topology of the quotient orbifold can be completely described by using  invariants and Proposition 2.10 in \cite{Dun1}. The formulae for these groups  are presented in Tables~\ref{pluto} and  \ref{pippo}. 

The remaining cases can be all analyzed by starting from the results obtained in Section~\ref{abelian} and by using a  procedure which is similar in each case.
In Section~\ref{remaining} we present  explicitly the most interesting families and we describe the general procedure. The formulae for  these  families are presented in Table~\ref{topolino}.  

In Subsection~\ref{last} we explain how it is possible, by the methods presented in \cite{Dun1}, to obtain in general  the singular set and the underlying topological space of the quotient orbifold from the invariants of the fibration. Here we treat the example of Family 2. For these groups we explicitly describe two different fibrations (with different base orbifolds) of the quotient orbifold. One of them derives from the Seifert fibration of the underlying manifold but the other does not.

\section{Classification of finite subgroups of $\SO4$}\label{algebraic-classification}

Let    $\H=\{a+bi+cj+dk\,|\,a,b,c,d\in\R\}=\{z_1+z_2j\,|\,z_1,z_2\in\C\}$ be the quaternion algebra.
In this section  we consider  the  3-sphere  as the set of unit quaternions (the quaternions of length 1):

$$
S^3=\{a+bi+cj+dk \,|\, a^2+b^2+c^2+d^2=1\}=\{z_1+z_2j\, \,|\, |z_1|^2+|z_2|^2=1\}
$$

The product in $\mathbb{H}$ induces a group structure on $S^3$.

\medskip

For each pair $(p,q)$ of elements of $S^3$,  the  function $\Phi_{p,q}:\H \rightarrow \H$  with $\Phi_{p,q}(h)=phq^{-1}$  leaves invariant the  length of  quaternions, thus  we can define  a homomorphism of groups $\Phi:S^3\times S^3\rightarrow \SO4 $ such that $\Phi(p,q)=\Phi_{p,q}$. The homomorphism can be proved to be surjective and the kernel of $\Phi$ is $\{(1,1),\,(-1,-1)\}$.
The homomorphism $\Phi$ gives a  1-1 correspondence between  finite subgroups of $\SO4$ and finite subgroups of $S^3\times S^3$ containing the kernel of $\Phi$. 
Moreover if two subgroups are conjugate in $\SO4$, then the corresponding groups in $S^3\times S^3$ are conjugate and vice versa.   
So to give a classification up to conjugation of the finite subgroups of $\SO4$, we consider the subgroups of $S^3\times S^3$. 

Let  $G$ be  a finite subgroup of $S^3\times S^3$, we denote by  $\pi_i:S^3\times S^3 \rightarrow S^3$ with $i=1,2$ the two projections. We use the following notations:  $L=\pi_1(G)$, $L_K=\pi_1((S^3\times\{1\})\cap G)$, $R=\pi_2(G)$, $R_K=\pi_2((\{1\}\times S^3)\cap G)$. The projection $\pi_1$ induces an isomorphism   $\bar{\pi}_1: G/(L_K\times R_K)\rightarrow L/L_K$ and $\pi_2$ induces an isomorphism   $\bar{\pi}_2: G/(L_K\times R_K)\rightarrow R/R_K$, we denote by $\phi_G$ the isomorphism between $L/L_K$ and $R/R_K$ obtained by composing  $\bar{\pi}_1^{-1}$ and $\bar{\pi}_2$. On the other hand if we consider $L$ and $R$,  two finite subgroups of $S^3$, with two normal subgroups $L_K$ and $R_K$ such that there exists an isomorphism $\phi:L/L_K\rightarrow R/R_K$, we can define a subgroup $G$ of $S^3\times S^3$ such that $L=\pi_1(G)$, $L_K=\pi_1((S^3\times\{1\})\cap G)$, $R=\pi_2(G)$, $R_K=\pi_2((\{1\}\times S^3)\cap G)$ and $\phi=\phi_G$.
The groups of $S^3\times S^3$ can be uniquely identified by 5-tuples $(L,L_K,R,R_K,\phi)$. The order of the group identified by $(L,L_K,R,R_K,\phi)$ is  $(|L|\cdot|R|)/|\,L/L_K|$, i.e. the product of the orders of $L$ and $R$ divided by the order of $L/L_K\cong R/R_K.$

We are interested  in the classification up to conjugacy and we use the following proposition (it is implicitly used in \cite{DV} and the proof is straightforward) :

\begin{Proposition}\label{classificationS3}
Let $G$ and $G'$ two groups of $S^3\times S^3$ described respectively by $(L,L_K,R,R_K,\phi)$ and  $(L',L'_K,R',R'_K,\phi')$.
The groups $G$ and $G'$ are conjugate in $S^3\times S^3$ if and only if there exist  two inner automorphisms, $\alpha$ and $\beta$, of $S^3$ such that $\alpha(L)=L'$, $\beta(R)=R'$, $\alpha(L_K)=L'_K$, $\beta(R_K)=R'_K$ and $\phi=\bar{\beta}^{-1}\phi'\bar{\alpha}$, where $\bar{\alpha}$ and $\bar{\beta}$ are the maps induced by $\alpha$ and $\beta$ on the factors $L/L_K$ and $R/R_K$.    
\end{Proposition}

Up to conjugacy the finite subgroups of $S^3$ are the following:

$$
\begin{array}  {rll}
C_n=& \{\cos(\frac{2\alpha\pi}{n})+i\sin(\frac{2\alpha\pi}{n}) \,|\,\alpha=0,\dots,n-1\}& \quad (n\geq 1)\\[5pt]
D^*_{4n}=& C_{2n}\cup C_{2n} j & \quad  (n\geq 2)\\[5pt]
T^*=&\cup_{r=0}^2 (\frac{1}{2}+\frac{1}{2}i+\frac{1}{2}j+\frac{1}{2}k)^r D^*_{8}&\\[5pt]
O^*=&T^*\cup (\sqrt{\frac{1}{2}}+\sqrt{\frac{1}{2}}j) T^*&\\[5pt]
I^*=&\cup_{r=0}^4 \left(\frac{1}{2}\tau^{-1}+\frac{1}{2}\tau j+\frac{1}{2} k\right)^r T^* & \quad \left(\tau=\frac{\sqrt{5}+1}{2}\right)\\[5pt]
\end{array}
$$

The group $C_n$ is cyclic of order $n$. 
The group $D_{4n}^*$ is a generalized quaternion group of order $4n$. The group $D_{4n}^*$ is called also binary dihedral and it is a central extension of the dihedral group by a group of order 2. 
The groups $T^*$, $O^*$ and $I^*$ are central extensions of the  tetrahedral, octahedral and icosahedral group, respectively, by a group of order two; they are called  binary tetrahedral, octahedral and icosahedral, respectively.

Analyzing the subgroups of the  groups in this list   and using  Proposition~\ref{classificationS3}, a classification  up to conjugation of the subgroups  of $S^3\times S^3$  containing $(-1,-1)$ can be given. We report these  groups in  Table~\ref{subgroup}. 
We remark that here we use the same enumeration of Du Val's list and almost the same notation. We decide for example to index a binary dihedral group with its order, while in \cite{DV} the group \textbf{D}$_n$ has order $4n.$ Another difference concerns the indices for the families  $1^\prime$ and $11^{\prime}$: what here is $r$  would be $2r$ in \cite{DV}.

\medskip

\begin{table}
\begin{tabular}{|l|c|c|c|}
\hline
 & family of groups & order of $\Phi(G)$ & \\
\hline
 1. & $(C_{2mr}/C_{2m},C_{2nr}/C_{2n})_s$ & $2mnr$ & $\gcd(s,r)=1$ \\  
 $1^{\prime}$. & $(C_{mr}/C_{m},C_{nr}/C_{n})_s$ & $(mnr)/2$ & $\gcd(s,r)=1$ $\gcd(2,n)=1$ \\
&&& $\gcd(2,m)=1$  $\gcd(2,r)=2$\\
 2. & $(C_{2m}/C_{2m},D^*_{4n}/D^*_{4n})$ & $4mn$ &  \\ 
 3. & $(C_{4m}/C_{2m},D^*_{4n}/C_{2n})$ & $4mn$ &   \\ 
 4. & $(C_{4m}/C_{2m},D^*_{8n}/D^*_{4n})$ & $8mn$ &   \\ 
 5. & $(C_{2m}/C_{2m},T^*/T^*)$ & $24m$ &   \\
 6. & $(C_{6m}/C_{2m},T^*/D^*_{8})$ & $24m$ &   \\ 
 7. & $(C_{2m}/C_{2m},O^*/O^*)$ & $48m$ &   \\
 8. & $(C_{4m}/C_{2m},O^*/T^*)$ & $48m$ &   \\ 
 9. & $(C_{2m}/C_{2m},I^*/I^*)$ & $120m$ &   \\ 
 10. & $(D^*_{4m}/D^*_{4m},D^*_{4n}/D^*_{4n})$ & $8mn$ &   \\
 11. & $(D^*_{4mr}/C_{2m},D^*_{4nr}/C_{2n})_s$ & $4mnr$ & $\gcd(s,r)=1$ \\ 
 $11^{\prime}$. & $(D^*_{2mr}/C_{m},D^*_{2nr}/C_{n})_s$ & $mnr$ &  $\gcd(s,r)=1$ $\gcd(2,n)=1$  \\
&&&  $\gcd(2,m)=1$ $\gcd(2,r)=2$\\
 12. &  $(D^*_{8m}/D^*_{4m},D^*_{8n}/D^*_{4n})$ & $16mn$ & \\
 13. &  $(D^*_{8m}/D^*_{4m},D^*_{4n}/C_{2n})$ & $8mn$ & \\
 14. &  $(D^*_{4m}/D^*_{4m},T^*/T^*)$ & $48m$ & \\
 15. &  $(D^*_{4m}/D^*_{4m},O^*/O^*)$ & $96m$ & \\
16. &  $(D^*_{4m}/C_{2m},O^*/T^*)$ & $48m$ & \\
17. &  $(D^*_{8m}/D^*_{4m},O^*/T^*)$ & $96m$ & \\
18. & $(D^*_{12m}/C_{2m},O^*/D^*_{8})$ & $48m$ & \\
19. & $(D^*_{4m}/D^*_{4m},I^*/I^*)$ & $240m$ & \\
20. & $(T^*/T^*,T^*/T^*)$ & $288$ & \\
21. & $(T^*/C_2,T^*/C_2)$ & $24$ & \\
$21^{\prime}.$ & $(T^*/C_1,T^*/C_1)$ & $12$ & \\
22. & $(T^*/D^*_{8},T^*/D^*_{8})$ & $96$ & \\
23. & $(T^*/T^*,O^*/O^*)$ & $576$ & \\
24. & $(T^*/T^*,I^*/I^*)$ & $1440$ & \\
25. & $(O^*/O^*,O^*/O^*)$ & $1152$ & \\
26. & $(O^*/C_2,O^*/C_2)$ & $48$ & \\
$26^{\prime}.$ & $(O^*/C_1,O^*/C_1)_{Id}$ & $24$ & \\
$26^{\prime\prime}.$ & $(O^*/C_1,O^*/C_1)_f$ & $24$ & \\
27. & $(O^*/D^*_{8},O^*/D^*_{8})$ & $192$ & \\
28. & $(O^*/T^*,O^*/T^*)$ & $576$ & \\
29. & $(O^*/O^*,I^*/I^*)$ & $2880$ & \\
30. &   $(I^*/I^*,I^*/I^*)$ & $7200$ & \\
31. &   $(I^*/C_2,I^*/C_2)_{Id}$ & $120$ & \\
$31^{\prime}.$ & $(I^*/C_1,I^*/C_1)_{Id}$ & $60$ & \\
32. &   $(I^*/C_2,I^*/C_2)_{f}$ & $120$ & \\
$32^{\prime}.$ & $(I^*/C_1,I^*/C_1)_f$ & $60$ & \\
33. & $(D^*_{8m}/C_{2m},D^*_{8n}/C_{2n})_f$ & $8mn$ &   $m\neq 1$  $n\neq 1$. \\ 
 $33^{\prime}$. & $(D^*_{8m}/C_{m},D^*_{8n}/C_{n})_f$ & $4mn$ &  $\gcd(2,n)=1 \gcd(2,m)=1$ \\
&&& $m\neq 1$ and $n\neq 1$.   \\
34. &  $(C_{4m}/C_{m},D^*_{4n}/C_{n})$ & $2mn$ & $\gcd(2,n)=1 \gcd(2,m)=1$ \\
\hline
\end{tabular}
\bigskip
\caption{Finite subgroups of $\SO4$}
\label{subgroup}
\end{table}

To be read the table needs some remarks:

\medskip

1. For most cases  the group is completely determined up to conjugacy  by the first four data in the 5-tuple  $(L,L_K,R,R_K,\phi)$, since any possible  isomorphism $\phi$  gives  the same group up to  conjugacy. So  we use   Du Val's notation where the group $(L,L_K,R,R_K,\phi)$ is denoted by $(L/L_K, R/R_K)$, using a subscript  only when the isomorphism has to be specified, as is the case for Families $1,\, 1^\prime,\,11,\, 11^\prime,\, 26^\prime,\, 26^{\prime\prime},\,31,\, 31^\prime,\,32,\, 32^\prime,\,33$ and  $33^\prime$. 
We recall that $\phi$ is an isomorphism  from $L/L_K$ to $R/R_K$.
In the group  $(C_{2mr}/C_{2m},C_{2nr}/C_{2n})_s$ the isomorphism is $\phi_s$ sending the element $(\cos(\pi/mr)+i \sin(\pi/mr))C_{2m}$ to $(\cos(\pi/nr)+i \sin(\pi/nr))^sC_{2n}$. In the group $(C_{mr}/C_{m},C_{nr}/C_{n})_s$ the situation is similar and the isomorphism is $\phi_s$ sending $(\cos(2\pi/mr)+i \sin(2\pi/mr))C_{m}$ to $(\cos(2\pi/nr)+i \sin(2\pi/nr))^sC_{n}$.

For Family 11 (resp. $11^\prime$) we extend the isomorphism $\phi_s$ by  sending the coset $jC_{2m}$ to $jC_{2n}$ (resp. $jC_m$ to $jC_n$).
If $L=D^*_{4mr}$, $R=D^*_{4nr}$, $L_K=C_{2m}$ and $R_K=C_{2n}$, then these isomorphisms cover all the possible cases except when $r=2$. 
In this case we have to consider another isomorphism $f:D^*_{4mr}/C_{2m}\rightarrow D^*_{4nr}/C_{2n}$ such that: 

$$ f((\cos(\pi/2m)+i \sin(\pi/2m) )C_ {2m})= jC_{2n} $$

$$ f(j C_ {2m})=(\cos(\pi/2n)+i \sin(\pi/2n))C_{2n} $$

This is due to the fact that, if $r>2$, the  quotients $L/L_K$ and $R/R_K$ are isomorphic to a dihedral groups of order greater then four where  the index-two cyclic subgroup is characteristic, while if $r=2$ the quotients are dihedral groups of order four and  extra isomorphisms appear.
The isomorphism $f$ gives another family of groups (number 33 in our 
list), which was missing  in Du Val's list. 

In Family $11^{\prime}$, the behavior is similar: if $r>2$, the isomorphisms $\phi_s$ give all the possible groups up to conjugacy; if $r=2$ the quotients are quaternion groups of order 8 and a further  family has to be considered. This is the second missing case in \cite{DV} and Family $33^{\prime}$ in  our list, where  $f$ is the following isomorphism:

$$ f((\cos(\pi/m)+i \sin(\pi/m) )C_ {m})= jC_{n} $$

$$ f(j C_ {m})=(\cos(\pi/n)+i \sin(\pi/n))C_{n}. $$

Families 33 and $33^{\prime}$ are listed in Table 4.2 of \cite{CS} as $\pm \frac{1}{4} [D_{4m}\times \overline{D}_{4n}]$ and $
\frac{1}{4} [D_{4m}\times \overline{D}_{4n}]$, respectively; their absence from Du Val's list was also  pointed out  in \cite{Dun-et-al}.

The groups in  Families $26^\prime,\, 26^{\prime\prime},\,31,\, 31^\prime,\,32,\, 32^\prime$ do not leave invariant any fibration of $S^3$, so we do not need an explicit description of the isomorphism in the  5-tuple.   More details can be found in \cite{DV} and in \cite{Dun3}.

2. The third family of groups not in  Du Val's list is Family 34 in Table~\ref{subgroup}. Note that $D^*_{4n}/C_{n}$ is cyclic of order 4 if and only if $n$ is odd. If $m$ is even while $n$ is odd, then $(C_{4m}/C_m,D^*_{4n}/C_n)$ does not contain  the kernel of $\Phi$, but if $m$ is odd, a new family is generated.  This family (after conjugation  by a reflection, and with the roles of $m$ and $n$ swapped) appears in Table 4.1 of \cite{CS} as $+\frac{1}{2} [D_{2m}\times C_{2n}]$.

\medskip

3. By Proposition~\ref{classificationS3}   the groups $(L,L_K,R,R_K,\phi)$ and $(R,R_K,L,L_K,\phi^{-1})$ are not conjugate unless $L$ and
$R$ are conjugate in $S^3$, so the corresponding groups in $\SO4$ are in general not conjugate in $\SO4$. 
If we consider conjugation in  $\O4$ the situation changes, because the orientation-reversing isometry of $S^3$, sending each quaternion $z_1+z_2 j$ to its inverse $\overline{z_1}-z_2j$, conjugates the two subgroups of $\SO4$ corresponding to $(L,L_K,R,R_K,\phi)$ and $(R,R_K,L,L_K,\phi^{-1})$. For this reason, in Table~\ref{subgroup} only one family between  $(L,L_K,R,R_K,\phi)$ and $(R,R_K,L,L_K,\phi^{-1})$ is listed.

\section{Seifert orbifolds}\label{Seifert orbifold}

We will describe the extension to 3-orbifolds of the concept of circle bundles over a surface, starting with brief general descriptions of 2-orbifolds and orientable 3-orbifolds. All 3-orbifolds considered in this paper will be orientable. For more details, see Chapter 2 of \cite{BMP}.

The underlying topological space of a 2-orbifold is a 2-manifold with boundary. If $x$ is a singular point,   a neighborhood of $x$  is modelled by $D^2/\Gamma$ where $\Gamma$ is a non-trivial finite group of isometries acting on $D^2$ fixing the preimage of $x$ in $D^2$. $\Gamma$ is also called the local group of $x$. The local group  can be a cyclic group of rotations ($x$ is called a cone point), a group of order 2 generated by a reflection ($x$ is a  mirror reflector) or a dihedral group  generated by an index 2 subgroup of rotations and a reflection  (in this case $x$ is called a corner reflector). The local models are presented in Figure~\ref{lm2o}, a cone point or a corner reflector is labelled by its singularity index, i.e. an integer corresponding   to the order of the subgroup of rotations in $\Gamma$. We remark that the  boundary of the underlying topological space consists of mirror reflectors and corner reflectors, and the singular set might contain  in addition some   isolated points corresponding to cone points. 
If $X$ is a 2-manifold  without boundary we denote by $X(n_1,\dots,n_k)$ the 2-orbifold with underlying topological space $X$ and with $k$  cone points  of singularity index  $n_1,\ldots,n_k$. If $X$ is a 2-manifold with  non-empty connected boundary  we denote by $X(n_1,\dots,n_k;m_1,\dots,m_h)$ the  2-orbifold with $k$ cone points of singularity index $n_1,\ldots,n_k$ and with $h$ corner reflectors of singularity index $m_1,\ldots,m_h$. All of our 2-orbifolds are covered by $S^2$, implying  $h\leq 3$; therefore all orderings of the corner reflectors are equivalent.

\begin{figure}[htb]
\begin{center}
\includegraphics[height=3.5cm]{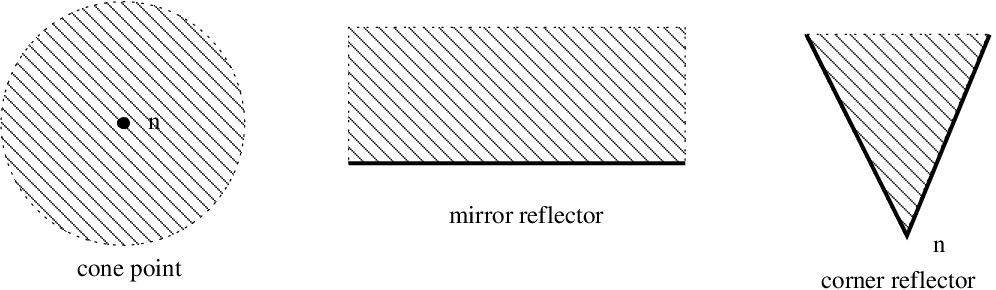}
\caption{Local models of 2-orbifolds}\label{lm2o}
\end{center}
\end{figure}

The underlying topological space of an  orientable 3-orbifold is a 3-manifold and the singular set is a trivalent graph. 
The local models are represented in Figure~\ref{lm3o}. Excluding the vertices of the graph, the local group of  a singular point  is cyclic;  an edge of the graph is  labelled by its singularity index, that is the order of the related cyclic local groups.

\begin{figure}[htb]
\begin{center}
\includegraphics[height=1.8cm]{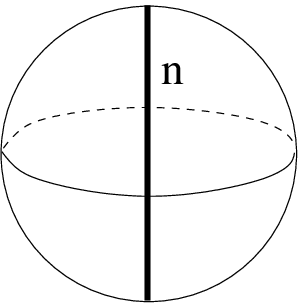}\hspace{0.7cm}\includegraphics[height=1.8cm]{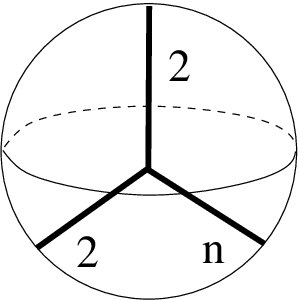}\hspace{0.7cm}\includegraphics[height=1.8cm]{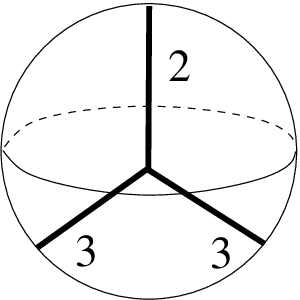}\hspace{0.7cm}\includegraphics[height=1.8cm]{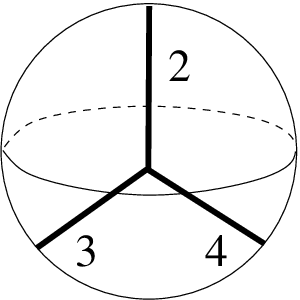}\hspace{0.7cm}\includegraphics[height=1.8cm]{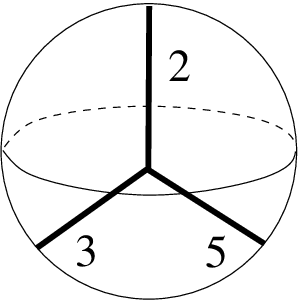}
\caption{Local models of 3-orbifolds}\label{lm3o}
\end{center}
\end{figure}

A Seifert fibration of a 3-orbifold $O$ consists of a projection $p:O\rightarrow B$, where $B$ is a 2-dimensional orbifold, such that for every point $x\in B$ there is an orbifold chart $x\in U\cong \tilde U/\Gamma$, an action of $\Gamma$ on $S^1$ (inducing a diagonal action of $\Gamma$ on  $\tilde U\times S^1$)  and a diffeomorphism $\psi:(\tilde U\times S^1)/\Gamma\rightarrow p^{-1}(U)$  which makes the following diagram commute:

\[
\xymatrix{
p^{-1}(U) \ar[dr]_-{p} & & (\tilde{U}\times S^1)/ \Gamma \ar[ll]_-{\psi} \ar[dl] & \tilde{U}\times S^1 \ar[l] \ar[dl]^-{\mbox{pr}_1} \\
 & U\cong\tilde{U}/ \Gamma  &  \tilde{U} \ar[l] &
}
\]

If we restrict our attention to orientable 3-orbifolds $O$, then the action of $\Gamma$ on $\tilde U\times S^1$ needs to be orientation-preserving. In this case, we will consider a fixed orientation both on $\tilde U$ and on $S^1$. Every element of $\Gamma$ may preserve both orientations, or reverse both.

The fibers $p^{-1}(x)$ are simple closed curves or intervals. If a fiber projects to a non-singular point of $B$, it is called generic. Otherwise we will call it exceptional.

Let us define the local models for an oriented Seifert fibered orbifold. Locally the fibration is given by the curves  induced on the quotient $(\tilde U \times S^1)/\Gamma$ by   the standard fibration of $\tilde U\times S^1$  given by the curves $\{y\}\times S^1$.

If the fiber is generic, it has a tubular neighborhood with a trivial fibration. 
When $x\in B$ is a cone point labelled by $b$, the local group $\Gamma$ is a cyclic group of order $b$ acting orientation-preservingly on $	\tilde U$ and thus it can act on $S^1$ by rotations. Suppose a generator of $\Gamma$ acts on $\tilde U$ by rotation of an angle ${2\pi}/{b}$ and on $S^1$ by rotation of $-{2\pi a}/{b}$. Then we define the local invariant associated to $x$ to be  $a/b$. This notation seems to suggest that $a/b$ represents a rational number but this not exact: different fractions giving the same rational number may represent different situations.  In fact in the orbifold context $a$ and $b$ are not necessarily coprime. The fiber $p^{-1}(x)$ may be singular (in the sense of orbifold singularities) and the  index of singularity is $\gcd(a,b)$. If $\gcd(a,b)=1$ the fiber is not singular.  A fibered neighborhood of the fiber $p^{-1}(x)$ is a fibered solid torus  (see Figure~\ref{cone-point}).   Forgetting the singularity of the fiber (if any), the local model  coincides with the  local model of a Seifert fibration for manifold (with invariant $(a/\gcd(a,b))/(b/\gcd(a,b))$).  

Note that, if $a\equiv a'(\mod b)$, the invariants $a/b$ and $a'/b$  describe the same situation.
In this section  the local invariants $a/b$ will be  normalized so that $0\leq a<b$. In the formulae we compute in Sections~\ref{abelian} and \ref{remaining} we give the local invariants in a non-normalized form. 

We remark that in the literature different sign conventions are used, we use the same as in \cite{BS} while in 
\cite{Dun1} the invariant is defined to be $-a/b.$
\begin{figure}[htb]
\begin{center}
\includegraphics[height=4cm]{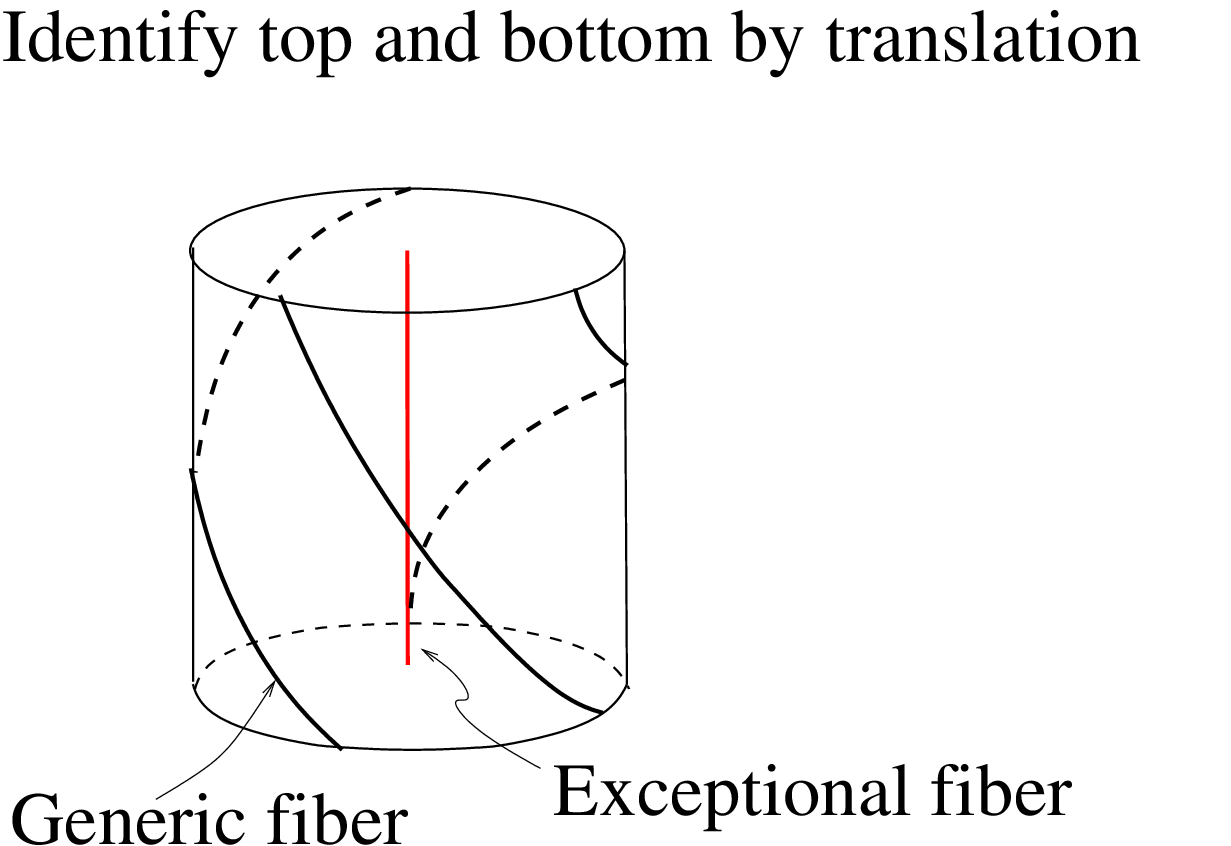}
\caption{A fibered neighborhood of an exceptional fiber of invariant 2/3 corresponding to a cone point}\label{cone-point}
\end{center}
\end{figure}

If $x$ is a corner reflector, namely $\Gamma$ is a dihedral group, then the non-central involutions in $\Gamma$ need to act on $\tilde U$ and on $S^1$ by simultaneous reflections. Here the local model is the so-called solid pillow, which is a topological 3-ball with some singular set inside.  $\Gamma$ has an index-two cyclic subgroup, acting as we previously described. The local invariant associated to $x$ is defined as the local invariant ${a}/{b}$ of the cyclic index-two subgroup. Again, the fiber $p^{-1}(x)$ has singularity index $\gcd(a,b)$. In Figure~\ref{corner-reflector}  the fiber $p^{-1}(x)$ is represented by the  vertical segment. The fibers of $\tilde U\times S^1$ intersecting the axes of reflections of $\Gamma$ in $\tilde U$ project to segments that are exceptional fibers of the 3-orbifold; the other fibers of $\tilde U\times S^1$  project to simple closed curves.  In Figure~\ref{corner-reflector} the  horizontal  segments are not fibers but consist of the endpoints of the fibers that are segments; they are singular (in the sense of orbifold singularities) of index two.

\begin{figure}[htb]
\begin{center}
\includegraphics[height=5cm]{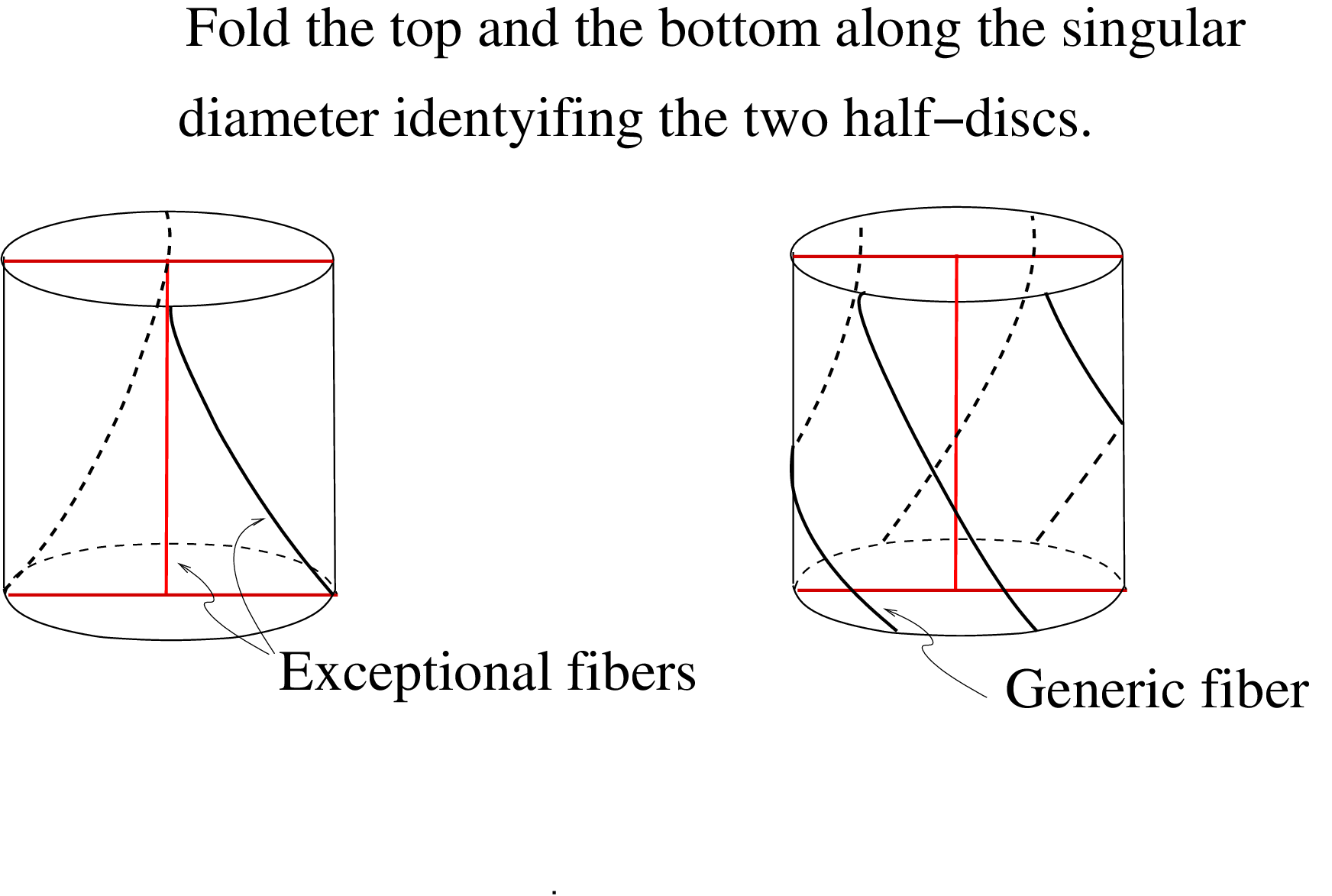}
\caption{Two copies of a fibered neighborhood of an exceptional fiber of invariant 1/2 corresponding to a corner reflector}\label{corner-reflector}
\end{center}
\end{figure}

Finally, over mirror reflectors (local group $\mathbb{Z}_2$), we have a special case of the dihedral case. The local model is topologically a 3-ball with two disjoint singular arcs of index 2. More details can be found in  \cite{BS} or \cite{Dun1}.

We will now state the classification theorem. An oriented  Seifert fibered orbifold will be determined, up to  diffeomorphisms which preserve the orientation and the fibration, by the following data: the base orbifold, the local invariants associated to cone points and corner reflectors, an additional invariant $\xi\in\mathbb{Z}_2$ associated to each boundary component of the base orbifold, and the Euler number. If we change the orientation of the orbifold, then the signs of local invariants and of the Euler number are also changed. The normalized local invariants pass in this case from $a/b$ to $(b-a)/b$. For the formal definitions of Euler number and of invariants associated to boundary components, as well as the proofs of the stated results, we refer again to \cite{BS} or \cite{Dun1}.

\begin{Theorem}\label{one}
Let $O$ and $O'$ be Seifert fibered orbifolds, where $O\rightarrow B$ and $O'\rightarrow B'$ are the fibration projections. If there is a diffeomorphism $\phi:B\rightarrow B'$, the Euler numbers $e(O)$ and $e(O')$ are equal and the local invariants associated to cone points, corner reflectors and boundary components of $B$ coincide with the local invariants of their images in $B'$ through $\phi$, then $O$ and $O'$ are  diffeomorphic.
\end{Theorem}

Since we will only need to consider base orbifolds arising as quotients of $S^2$, and these have at most one boundary component, the formula in Proposition 2 implies that the local invariant of that component is forced by the other data. We will make frequent use of the following corollary of Theorem \ref{one}.

\begin{Corollary}
Let $O$ and $O'$ be Seifert fibered orbifolds, where $O\rightarrow B$ and $O'\rightarrow B'$ are the fibration projections. Suppose $B$ and $B'$ have at most one boundary component. If there is a  diffeomorphism $\phi:B\rightarrow B'$, the Euler numbers $e(O)$ and $e(O')$ are equal and the local invariants associated to cone points and corner reflectors of $B$  coincide with the local invariants of their images in $B'$ through $\phi$, then $O$ and $O'$ are  diffeomorphic.
\end{Corollary}

The following statements will be useful.

\begin{Theorem}\label{eulero}
Let $\pi:O\rightarrow O'$ be a finite orbifold covering, where $O\rightarrow B$ and $O'\rightarrow B'$ are the fibration projections. Suppose $\pi$ preserves the fibrations and thus induces an orbifold covering $\bar \pi:B\rightarrow B'$ of degree $l$. Moreover, suppose $m$ is the degree with which a generic fiber of $O$ covers its image in $O'$ (note that $lm$ is the degree of the covering $\pi$). Then the Euler numbers of $O$ and $O'$ are in the following relation: $$e(O)=\frac{m}{l}e(O')\,.$$
\end{Theorem}

\begin{Proposition} \label{somma eulero invarianti}
Let $O$ be a Seifert fibered orbifold with Euler number $e(O)$ and local invariants associated to cone points, corner reflectors and boundary components of the base orbifold respectively ${a_t}/{b_t}$, ${a_{sk}}/{b_{sk}}$ and $\xi_k$. Then $$e(O)+\sum_t\frac{a_t}{b_t}+\frac{1}{2}\sum_k\left(\sum_s\frac{a_{sk}}{b_{sk}}+\xi_k\right)\equiv 0\;\textnormal{mod}1\,.$$
\end{Proposition}

\bigskip

Seifert fibrations of $S^3$ are well known: it is proved in \cite{SeT} that, up to diffeomorphism, they are  given by  the maps of the form 
$\pi:S^3\rightarrow S^2\cong \mathbb{C}\cup\left\{\infty\right\}$ $$\pi(z_1+z_2 j)=\frac{z_1^u}{z_2^v}\qquad\textrm{or}\qquad\pi(z_1+z_2 j)=\frac{\overline{z}_1^u}{z_2^v}$$ for $u$ and $v$ coprimes. The base orbifold is $S^2$ with two possible cone points. When $u=v=1$, $\pi(z_1+z_2 j)={z_1}/{z_2}$ is called the Hopf fibration. In this case the base orbifold is $S^2$ and all the fibers are generic;  if we consider  the  orientation of $S^3$ induced by the standard orientation of $\C \times \C$, the  Euler number of the Hopf fibration is $-1$.

It is known (see Theorem 5.1 in \cite{DM}) that an orientable   Seifert fibered  3-orbifold $S^3/G$ is  isometric to an orbifold $S^3/G'$ where $G'$ is a subgroup of  $\SO4$ respecting the Hopf fibration; the isometry may be orientation-reversing.

It can be easily checked that the isometry corresponding to $(0,w_1+w_2 j)\in S^3\times S^3$ does preserve the Hopf fibration, with induced action on $S^2$ given by $$\lambda\mapsto \frac{\overline w_1 \lambda+\overline w_2}{-w_2\lambda+w_1}\,.$$
Analogously, it can be checked that an isometry given by $(w_1+w_2 j,0)$ preserves the Hopf fibration if $w_1=0$ or $w_2=0$, but not in the general case. 
On the other hand, the general fibration $\pi(z_1+z_2 j)={z_1^u}/{z_2^v}$ is preserved by $(w_1+w_2 j,u_1+u_2 j)$ provided $w_2=u_2=0$ or $w_1=u_1=0$. It is not necessary to repeat the computations for the remaining fibrations; it suffices to note that the orientation-reversing isometry $z=z_1+z_2 j\mapsto z^{-1}=\overline z_1-z_2 j$ maps the fibration $\pi(z_1+z_2 j)={\overline{z}_1^u}/{z_2^v}$ to $\pi(z_1+z_2 j)={z_1^u}/{z_2^v}$.  The isometry $\Phi(w,w')$ preserves a fibration $\pi(z_1+z_2 j)={\overline{z}_1^u}/{z_2^v}$ if and only if $\Phi(w'^{-1},w^{-1})$ preserves $\pi(z_1+z_2 j)={z_1^u}/{z_2^v}$.

If $L$ is  $C_n$ or $D^*_{2n}$ the group $G=(L,L_K,R,R_K,\phi)$ leaves invariant the Hopf fibration. 
We note that if  $G$ is  a group of  Families $1,$ $1',$ $11,$ and $11'$, then $G$  preserves all the fibrations of the sphere and  the quotient orbifold $S^3/G$ admits infinite fibrations.
When $R$ is  $C_n$ or $D^*_{2n}$  the fibration given by $\frac{\overline{z}_1}{z_2}$ is left invariant and the group can be conjugated by the orientation-reversing map $z\rightarrow z^{-1}$ to the group $(R,R_K,L,L_K,\phi^{-1})$ which preserves the Hopf fibration (as one should expect by Theorem 5.1  in \cite{DM}).
If both $L$ and $R$ are isomorphic to $T^*$, $O^*$ or $I^*$ no fibration of $S^3$ is preserved (these are the groups considered in \cite{Dun3}).

\section{The quotient of $S^3$ by an abelian or generalized dihedral group}\label{abelian}

We consider first the quotients of $S^3$ by the groups that are images under $\Phi$ of  groups belonging to Families 1, $1'$, $11$ and $11'$. The groups in this family are abelian or generalized dihedral. We compute the fibration induced on the quotient by the Hopf fibration of $S^3$.

The 3-sphere  $S^3=\{z_1+z_2j\, \,|\, |z_1|^2+|z_2|^2=1\}$ can be decomposed by two solid tori $T_1=\{z_1+z_2j\,|\, |z_1|\leq \sqrt {2}/2\}$ and $T_2=\{z_1+z_2j\,|\, |z_2|\leq \sqrt{2}/2\}$. 
We consider first  the  isometries  of $S^3$ sending $z_1+z_2j$ to $w_1 z_1+w_2 z_2 j$ where $w_1$ and $w_2$ are fixed complex numbers of norm 1. The isometries of this kind leave invariant the two solid tori. We denote by  $\mu$ and $\lambda$ two oriented curves in the common boundary of the two solid tori such that $\mu$ is the meridian of $T_1$ and $\lambda$ is the meridian of $T_2$ (and a longitude of  $T_1$).   

\begin{Lemma}\label{quoziente-toro}
Let $\rho$ be the isometry of $S^3$ sending $z_1+z_2j$ to $e^{2\pi\frac{g}{f}i}z_1+e^{2\pi\frac{d}{f}i}z_2j$ where $d,\,f,\,g$ are integers such that $\gcd(d,\,f,\,g)=1$ and $f\neq 0$. 
The  quotient of $T_1/\langle\rho\rangle $  is again a solid torus. Moreover we can choose a pair of oriented curves $(\mu',\lambda')$    in  the boundary  of $T_1/\langle\rho\rangle $ such that $\mu'$ is a meridian and  $\lambda'$ a longitude  and such that $\pi_1^*$, the map induced by the quotient map $\pi_1: T_1 \rightarrow T_1 /\langle\rho\rangle $ on the first homology group of the boundary, acts in the following way:

$$
\begin{array}{rl}
\pi_1^*([\mu])&=\gcd (d,f)[\mu']\\
\pi_1^*([\lambda])&=-(g\overline{d'})[\mu']+f' [\lambda']
\end{array}
$$

\noindent
 where $d'=d/ \gcd(d,f)$, $f'=f/ \gcd(d,f)$ and $\overline{d'}$ is the positive integer strictly smaller then $f'$ such that $d'\overline{d'}\equiv 1\,\mod f'.$ 
\end{Lemma}

\begin{proof}

We can think at $T_1$ as a cylinder of height $2\pi$ with the top and the bottom identified by a translation. To obtain  a  fundamental domain, we can first cut  the cylinder along a plane parallel to the bases obtaining  a smaller cylinder of height $2\pi/f'$, then we take a wedge of the smaller cylinder  with  angle equal to  $2\pi/\gcd(d,f)$. The quotient $T_1/\langle\rho\rangle $ can be visualized by  identifying the lateral sides of this fundamental domain with a rotation (which gives again a cylinder), and then  identifying the bottom disk to the top disk of this new cylinder with a twist of $-(g\overline{d'}2\pi)/f'$ radians. The  angle of the twist is  computed using $\rho^{\overline{d'}}$, a power of $\rho$ acting as a $2\pi/f'$ translation along the height of the starting cylinder. 

This representation of the quotient ensures that $T_1/\langle\rho\rangle $ is a torus and makes evident that a meridian of $T_1$ is sent by $\pi_1$ to $\gcd(d,f)$ times a meridian of $T_1/\langle\rho\rangle $.  With the choice of the appropriate longitude $\lambda'$ in the quotient (see for example \cite[p.362-363]{SeT}) we obtain also   $\pi_1^*([\lambda])=-(g\overline{d'})[\mu']+f' [\lambda'].$

\end{proof}

We will use this lemma to compute the invariant of the fibration induced on the quotient torus by the fibration of $T_1$.  We remark that the core of $T_1$ is fixed pointwise by a subgroup of $\langle\rho\rangle $ of order $\gcd(d,f)$.  If the torus $T_1$ is fibered by $p\mu+q\lambda$ curves, the quotient torus is fibered by $(\gcd(d,f)p-g\overline{d'}q)\mu'+f'q\lambda'$ curves, the slope of the fiber is $(\gcd(d,f)p-g\overline{d'}q)/qf'$ (in this case the fraction might be reducible). We consider the corresponding reduced fraction $a/b$ and  an integer $\overline{a}$ such that $a\overline{a}\equiv 1\, \mod b$, and  find  that the local invariant equals $(\overline{a}\gcd(d,f))/(b\gcd(d,f))$(see~\cite[p.364]{SeT} or~\cite[p.37]{BS}).

We note that, on the one hand,  the choice of $\lambda'$ and the choice of integer $\overline{d'}$ do not affect the value of the invariant, while on the other hand,  the invariant  depends on the homology class of the fiber in $T_1$ and not only on the invariant of the fibration of $T_1$ (that can be normalized mod 1). 

If we consider the quotient of  $T_2$ by  $\langle\rho\rangle $, we can obtain an analogous lemma where the roles of $d$ and $g$ are exchanged.

\begin{Lemma}\label{rel-orbifold-base}
Let $\OO$ be a fibered orbifold with base orbifold $B$ and let $G$  be a finite group of orientation-preserving diffeomorphisms of $\OO$ preserving the fibration. The group $G$ acts  on $B$ (but not necessarily effectively).   The quotient orbifold  $\OO/G$ is fibered by the images of fibers of $\OO$ and the base orbifold of  $\OO/G$ is $B/G$.
\end{Lemma}
\begin{proof}

We recall that if  an element of  $G$ maps a fiber to another fiber, the two fibers have the same local model.

We work locally using fibered neighborhoods of the fibers. 

If a fiber $\alpha$ is not invariant under the action of any non-trivial element of $G$, the group $G$ acts freely on  the $|G|$ fibers in the orbit of $\alpha$, and the quotient map $\pi:\OO\rightarrow\OO/G$ can be restricted to a fibered neighborhood of $\alpha$ obtaining a diffeomorphism preserving fibers.   The situation of the base orbifold reflects exactly  this behavior.

We now consider the case in which   $\alpha$ has a non-trivial stabilizer in $G.$ 

Suppose first that $B$ has only  cone points,  so that  $\alpha$  is a  simple closed curve with a fibered tubular neighborhood (we will only need to use this case of Lemma \ref{rel-orbifold-base} in what follows). Afterwards, we will explain how to complete the proof in the general case when $B$ has mirror reflectors, and possibly has corner reflectors.

The quotient map can be restricted locally to an orbifold covering map $\pi:D^2\times S^1\rightarrow (D^2\times S^1)/G_0$ where $G_0$ is the stabilizer of $\alpha$.  We consider in the boundary of $D^2\times S^1$ a longitude $\lambda$ and a meridian $\mu$; let   $p$ and $q$ be  the coprime integers such that a generic fiber of $D^2\times S^1$ is homologous to $p\mu+q\lambda$. The invariant of the fiber $\alpha$ is $\overline{p}n/qn$, where $n$ is the index of singularity of $\alpha$ and $\overline{p}p\equiv 1\,\mod q$. 

We suppose first that $G_0$ fixes  $\alpha$ pointwise. An argument similar to that of Lemma~\ref{quoziente-toro} proves that $(D^2\times S^1)/G_0$ is a solid torus fibered by the images of the fibers of $D^2\times S^1$; the image of a generic fiber of $D^2\times S^1$ is homologous to  $ph\mu'+q\lambda'$, where $\mu'$ and $\lambda'$ are meridian and longitude of the quotient torus and $h$ is the order of $G_0$. The slope of the image of a generic fiber is $(ph)/q$ (where $ph$ and $q$ need not be coprime). In particular  the base orbifold of $ (D^2\times S^1)/G_0$ is a disk and the cone point, that is the image of $\alpha$, has singularity index $(nhq)/\gcd(h,q)$.
 To compute the  action of $G_0$ on the base orbifold,  we consider that the general fiber intersects $q$ times a transverse disk bounded by a meridian  and a circular sector of the transverse disk of angle $2\pi/q $ intersects all the fibers. So $G_0$ induces a group of rotations of order $h/\gcd(h,q)$ fixing the point corresponding to $\alpha$. Since the  cone point in $B$ corresponding to $\alpha$  has index $nq$, in the quotient $B/G$ the singular point has index $(nqh)/\gcd(h,q)$, matching the situation of the base orbifold of $ (D^2\times S^1)/G_0$.

We can suppose now that the group $G_0$ acts effectively on the fiber $\alpha$; if not  we can consider the quotient of $G_0$ by the normal cyclic subgroup of elements fixing pointwise $\alpha$. Since $G_0$ acts effectively on a 1-sphere,  it is cyclic or  dihedral. We consider first the normal cyclic subgroup $G_1$ of elements which act  preserving the orientation on $\alpha$. This group is generated by the map $\rho$ sending $z_1+z_2 j$ to $e^{2\pi\frac{g}{f}i}z_1+e^{2\pi\frac{d}{f}i}z_2j$; since the action on $\alpha$ is effective, we have that $\gcd(d,f)=1$.  The proof of Lemma~\ref{quoziente-toro} implies  that $(D^2\times S^1)/G_1$ is a solid torus fibered by the images of the fibers of $D^2\times S^1$; the image of the generic fiber of $D^2\times S^1$ is homologous to  $(p-qg\overline{d})\mu'+q f \lambda'$.  The index of the cone point in the base orbifold of $(D^2\times S^1)/G_1$ is $(nqf)/\gcd(p-qg\overline{d},f)$.

To compute the action of $\rho$ induced on the base orbifold, we consider a meridian disk $D=D^2\times\{*\}$ in $D^2\times S^1$.   The map $\rho$ rotates the $D^2$-coordinate of this disk by $2\pi g/f$ radians; if  the fibers are homologous to the longitude this is exactly the rotation induced on $B$.  In general  we have to consider that the generic fibers connect the points of $D$ and $\rho(D)$ with a rotation of $-(2\pi pd)/(qf)$ (see Figure \ref{proof}). Moreover  a generic fiber intersects  the meridian disk $q$ times, so the angle of the rotation induced on the base orbifold has to be multiplied by $q$. Therefore, $\rho$ induces a rotation on the base orbifold of angle $2\pi(qg -p d)/f$ fixing the cone point corresponding to $\alpha$. Since $f$ and $\overline{d}$ are coprime we can replace the  angle with  $2\pi(qg\overline{d}-p)/f$. Considering that the cone point of $B$ has index $nq$, the cone point of $B/G$ has the same singularity index  as the cone point of the base orbifold of $(D^2\times S^1)/G_1$.

\begin{figure}[htb]
\begin{center}
\includegraphics[height=4cm]{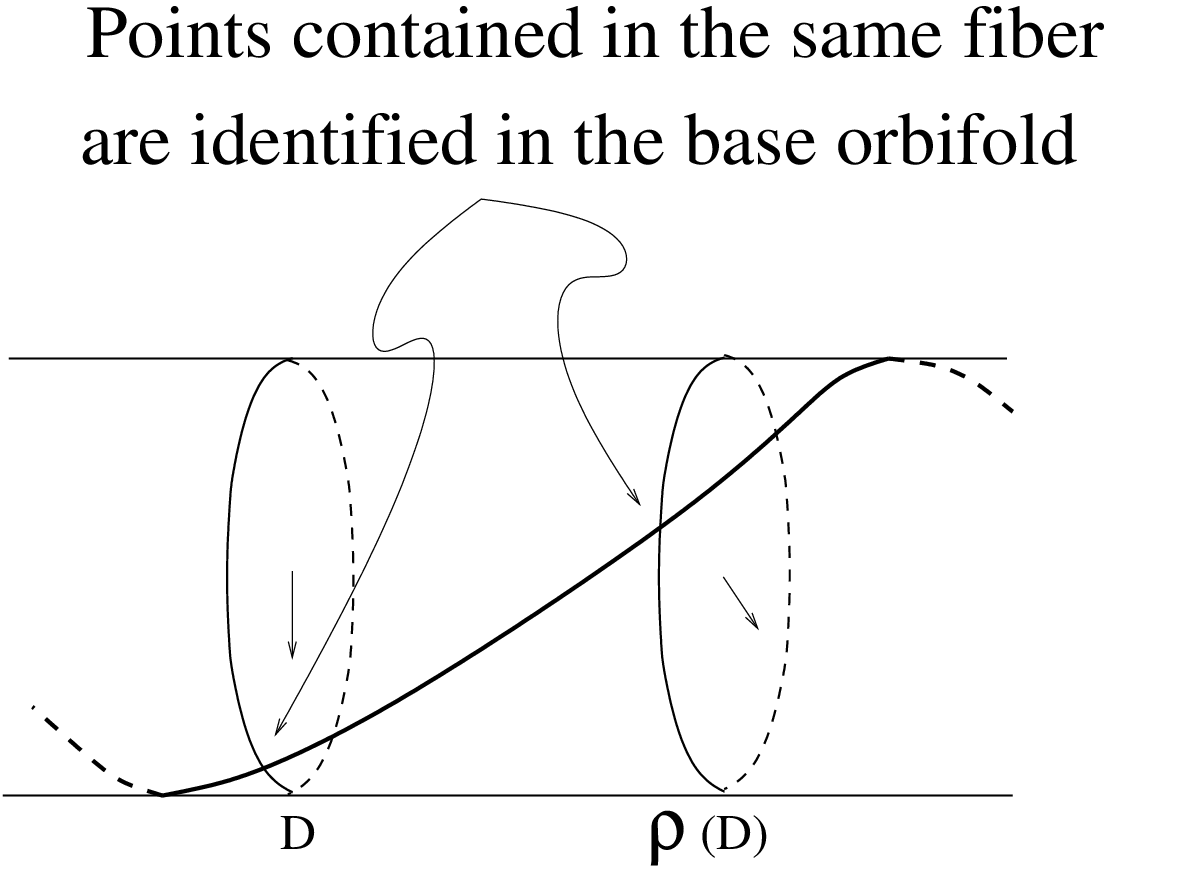}
\caption{Action of $\rho$}\label{proof}
\end{center}
\end{figure}

If $G_0$ is different from $G_1$ we have to consider a further quotient  passing from $(D^2\times S^1)/G_1$ to $((D^2\times S^1)/G_1)/(G_0/G_1)\cong (D^2\times S^1)/G_0$. 
In this case the quotient is a fibered solid pillow and the induced action on the base orbifold of $(D^2\times S^1)/G_1$ is by reflection. It is clear that the quotient of the base orbifold and the base orbifold of the quotient coincide.

The proof can be completed for the general case by considering  actions of involutions on solid pillows.  
In fact, if $\alpha$ is a fiber corresponding to a mirror reflector or to a corner reflector in $B$,  then  the stabilizer of $\alpha$ in $G$ preserves a regular neighborhood of $\alpha$ which is a solid pillow (recall Figure \ref{corner-reflector}). All elements of the  stabilizer leave invariant both $\alpha$ and the four singular points on the boundary of the pillow (these points are connected, in pairs, to the endpoints of $\alpha$). Hence, keeping in mind that $G$ preserves orientation, whenever an element does not act trivially, its square must act trivially.   We have two possibilities:  the involution either fixes  $\alpha$ pointwise or exchanges its endpoints. If the involution exchanges the two endpoints of $\alpha$  we have two possible actions on the singular points of the boundary of the solid pillow. Distinguishing these three cases we can consider the action of the stabilizer of $\alpha$ on the solid pillow and the action induced on the base orbifold. For each of the three cases the situation is different according to the parity of $p$ and $q$ where $p/q$ is the slope of the generic fiber in the solid pillow (here $p$ and $q$ are considered coprime). Finally we have to analyze nine cases, but in each we find that   the quotient of the base orbifold and the base orbifold of the quotient coincide.

\end{proof}

We consider in $S^3$  the curves  $\alpha_1=\{e^{it}j\,|\, t\in[0,2\pi]\}$  and   $\alpha_2=\{e^{it}\,|\, t\in[0,2\pi]\}$, they are the cores of $T_1$ and $T_2$ respectively and fibers of every fibration of $S^3$ described in Section~\ref{Seifert orbifold}.

We consider first some  particular cases, then we proceed to analyze  Family $1'$ and Family 1 in general; for  Family 1 we have to distinguish some subcases depending on the parity of certain indices.
The results about the quotients of $S^3$ by groups in  Family $1'$ are summarized in Table~\ref{pluto}
while the results for  Family 1 can be found in  Table~\ref{pippo}. Finally we will consider Families 11 and $11'.$

\bigskip

\noindent
\textbf{\large{Case 1}}.  $G=\Phi((C_{2h}/C_1,C_{2h}/C_1)_1)$

\medskip

This group is generated by the map sending $(z_1+z_2j)$ to $z_1+e^{2i\pi/h}z_2j$ that fixes $\alpha_2$  pointwise and is the rotation around $\alpha_2$ of angle $(2\pi)/h$. The underlying topological space of the quotient orbifold $S^3/G$ is again a three-sphere, and  the singular set of the orbifold is the image of $\alpha_2$ (a trivial knot), with singularity index equal to $h$. $G$ preserves the Hopf fibration of $S^3$ (as defined in Section \ref{Seifert orbifold}) and the images of the fiber give a fibration of $S^3/G.$ 
Applying  Lemma~\ref{quoziente-toro} to the tori $T_1$ and $T_2$ we can see that the fibration of $S^3/G$  has an exceptional (and not singular)  fiber with invariant $1/h$ (the image of $\alpha_1$) and a singular  fiber  of singular index $h$ with $0/h$ as  invariant (the image of $\alpha_2$). The base orbifold  is a 2-sphere with two cone points of index $h$. By Theorem~\ref{eulero} the Euler number is $-1/h.$ 

\bigskip

\noindent
\textbf{\large{Case 2}}.  $G=\Phi((C_{2h}/C_1,C_{2h}/C_1)_{-1})$

\medskip

Inverting the  roles of $\alpha_1$ and $\alpha_2$, the situation  is analogous to the previous case.

\bigskip

\noindent
\textbf{\large{Case 3}}.  $G=\Phi((C_{2h}/C_{h},C_{2h}/C_{h})_1)$
\medskip

The group $G$  is isomorphic to $\Z_h\times \Z_h$ and it is generated by $\rho_1$, the map sending $(z_1+z_2j)$ to $e^{2i\pi/h}z_1+z_2j$,  and $\rho_2$ the map sending  $(z_1+z_2j)$ to $z_1+e^{2i\pi/h}z_2j$. The fiber $\alpha_1$ (resp. $\alpha_2$) is fixed pointwise by $\rho_1$ (resp. $\rho_2$). We can analyze $S^3/G$ by considering successive quotients, the first one by the group generated by $\rho_1$ and the second one by the group generated by the projection of $\rho_2$ to the quotient $S^3/\langle\rho_1\rangle $. Since we quotient by cyclic groups  generated by an element with non-empty fixed point set, the underlying topological space of $S^3/G$ is again a 3-sphere and  the singular set is a link with two components, both of index $h$. 
 We consider in $S^3$ the Hopf fibration (the fiber is of type (1,1)), $G$ preserves this fibration and the images 
of the fibers in $S^3/G$ give a fibration. Using  Lemma~\ref{quoziente-toro}, we obtain  that the induced fibration is again the Hopf fibration (with two singular fibers), the base orbifold is a 2-sphere with two cone points of index $h$ and the Euler number is $-1$.  

\bigskip
\noindent
\textbf{\large{Case 4}}. $G=\Phi((C_{mr}/C_{m},C_{nr}/C_{n})_s)$ (Family $1'$)
\medskip

In this family all the groups are abelian.

We recall that in this case $m$ and $n$ are odd integers, $r$ is even and  $s$ is coprime with $r.$

\begin{Claim}\label{base-euler}
The base orbifold of the  fibration of $S^3/G$ induced by the Hopf fibration of $S^3$  is  $S^2$ with two cone points of index $nr/2$. The Euler number of the  fibration of $S^3/G$ is $-2m/nr.$
\end{Claim}

\begin{proof}
By Lemma~\ref{rel-orbifold-base} the  base orbifold of $S^3/G$ is the quotient of the base orbifold of the Hopf fibration (a 2-sphere) by the action induced by  $G$.
Using the formulae given in Section~\ref{Seifert orbifold} we can compute that the action induced by $G$ on the 2-sphere   corresponds to the action of a cyclic group of rotations of order $nr/2$, fixing the two points that are the images of the two fibers $\alpha_1$ and $\alpha_2$. This proves the first part of our statement. The action of $G$ on the base orbifold shows that the generic fiber of $S^3/G$ is covered by  $nr/2$ distinct fibers  of $S^3$; since the order of $G$ is $mnr/2$ each fiber of $S^3$ is an  $m$-fold cover of its image in $S^3/G.$ By Theorem~\ref{eulero}, the Euler number is $-2m/nr.$
\end{proof}

To complete the description of  the fibration of $S^3/G$, we have   to compute the invariants of the exceptional fibers.

\begin{remark}\label{coprime}
In the computation  of the invariants we  suppose that $m$ and $n$ are coprime. In fact if $\gcd(m,n)=h>1$,  the group $G$ contains the subgroup $G_0=\Phi((C_{2h}/C_{h},C_{2h}/C_{h})_1)$. 
The quotient  $S^3/G_0$ is again the 3-sphere with the Hopf fibration, but the images of $\alpha_1$ and $\alpha_2$ are singular of index $h$ (see Case 2). This implies that the local invariants of $S^3/G$ can be obtained from the invariants  of $S^3/(G/G_0)$ with the fibration induced by the Hopf fibration.  They  coincide except for those of $\alpha_1$ and $\alpha_2$. For these two fibers the numerator and the denominator of the invariant have to be multiplied by $h$.  The action of $G/G_0$ on $S^3$ coincides with the action of $G=\Phi((C_{m'r}/C_{m'},C_{n'r}/C_{n'})_s)$ where $m'=m/h$ and $n'=n/h$. To simplify the notation  we  suppose that $m$ and $n$ are coprime and at the end of the process it will be  enough to replace $m$ and $n$ with $m'$ and $n'$ and to multiply both numerator and denominator of the local invariants of the exceptional fibers by a factor of  $h.$  
\end{remark}

We now introduce some notation: 

Let $\phi$ be the map sending $z_1+z_2j$ to $e^{\left(2i\pi\frac{n-sm}{mnr}\right)}z_1+e^{\left(2i\pi\frac{n+sm}{mnr}\right)}z_2j$ and $\gamma$ the map sending $z_1+z_2j$ to $e^{\left(-2i\pi\frac{1}{n}\right)}z_1+e^{\left(2i\pi\frac{1}{n}\right)}z_2j$. By construction of $G$, the maps  $\phi$ and $\gamma$  together with the map sending $z_1+z_2j$ to $e^{\left(2i\pi\frac{1}{m}\right)}z_1+e^{\left(2i\pi\frac{1}{m}\right)}z_2j$ generate $G$; this last map equals $\phi^r\gamma^{-s}$, so $\phi$ and $\gamma$ are enough to generate $G$.
We denote  $a=\gcd(n+sm,n-sm,mnr)$,  $b_1=\gcd(\frac{n-sm}{a},\frac{mnr}{a})$ and $b_2=\gcd(\frac{n+sm}{a},\frac{mnr}{a})$.

\begin{remark}\label{relation1}
It is easy to see that $a$ and $m$ are coprime, from which it follows that $a=\gcd(n+sm,n-sm,nr)$ and $a/2=\gcd(n,s)$. We remark that  $b_1$ and $b_2$ are coprime.

From  $a=\gcd(n+sm,n-sm,nr)$ we deduce that $(n-sm)/a$ and $2n/a$ are coprime, so we obtain that  $b_1=\gcd(\frac{n-sm}{a},\frac{mr}{2})$. Since $a/2=\gcd(n,s)$, the integer $a/2$ is coprime with $mr$, so we get $2b_1=\gcd(n-sm,mr)=\gcd(n-sm,r).$
 Analogously the equality $2b_2=\gcd(n+sm,r)$ can be  obtained.
\end{remark}

\begin{Claim}\label{fixed-point-subgroup}
 The subgroup of $G$ generated by the elements with non-empty fixed point set is generated by $\phi^\frac{mnr}{ab_1}$ and $\phi^\frac{mnr}{ab_2}.$
\end{Claim}

\begin{proof}

We note that $ab_1=\gcd(n-sm,mnr)$ and $ab_2=\gcd(n+sm,mnr)$. Using these equalities it is easy to see 
that $\phi^\frac{mnr}{ab_1}$ fixes  $\alpha_2$ pointwise and, since $(n+sm)/a$ and $b_1$ are coprime,  acts as a rotation of order $b_1$ on $\alpha_1$. The rotation $\phi^\frac{mnr}{ab_2}$ fixes  $\alpha_1$ pointwise and, since $(n-sm)/a$ and $b_2$ are coprime,  acts as a rotation of order $b_2$ on $\alpha_2$.

Thinking of the elements of $G$ as matrices in $\SO4$, it is easy to see that   $\alpha_1$ or $\alpha_2$ 
are the only fibers that can be fixed pointwise by an element of $G$.
Consider an element $\phi^t \gamma^u$ fixing pointwise $\alpha_1$, in this case the order of the action of the element on $\alpha_1$  has to be one. Since the order of the action of $\phi$ on $\alpha_1$ is $(mnr)/(b_2a)$, the order of $\phi^t$ on $\alpha_1$ is $(mnr)/(ab_2k)$ where $k=\gcd((mnr)/(b_2a),t);$ also,  the order of $\gamma$ on $\alpha_1$ is $n.$
If the action of $\phi^t \gamma^u$ is trivial on $\alpha_1$, then the integer $(mnr)/(ab_2k)$ is a divisor of $n$. This implies that $mr$ divides $ab_2k=(a/2)(2b_2k)$; using  $\gcd(a/2,m)=1$, $\gcd(a/2,r)=1$  and  $b_2=\gcd((n+sm)/2,r/2)$, we obtain that $(mr)/(2b_2)$ divides $k$ and consequently divides $t$. Moreover the action of $\phi^t \gamma^u$ on $\alpha_1$ is given by the multiplication of $z_2$ by the element $e^{\left(t\frac{n+sm}{mnr}+\frac{u}{n}\right)2\pi i}$, so  $t\frac{n+sm}{mnr}+\frac{u}{n}$ is an integer.
The action on $\alpha_2$ is given by the  multiplication  of $z_1$ by the element:

$$ e^{\left(t\frac{n-sm}{mnr}-\frac{u}{n}\right)2\pi i}=e^{\left(t\frac{n-sm}{mnr}+t\frac{n+sm}{mnr}-t\frac{n+sm}{mnr}-\frac{u}{n}\right)2\pi i}=e^{\left(t\frac{n-sm}{mnr}+t\frac{n+sm}{mnr}\right)2\pi i}=e^{\left(\frac{2tn}{mnr}\right)2\pi i}$$

Since  $(mr)/(2b_2)$ divides  $t$ we obtain that $\frac{2tn}{mnr}=\frac{t'}{b_2}$ where $t'$ is an integer. This implies   that $\phi^t \gamma^u$ is a power of $\phi^\frac{mnr}{ab_2}.$ 

Analogously we can prove that any element fixing pointwise $\alpha_2$ is a power of $\phi^\frac{mnr}{ab_1}.$

\end{proof}

\begin{Claim} 
The quotient orbifold  $S^3/\langle \phi^\frac{mnr}{ab_1},\phi^\frac{mnr}{ab_2}\rangle $ has a 3-sphere as  underlying topological space. Only the projections  of the fibers $\alpha_1$ and $\alpha_2$ might be  singular, respectively,  of  singularity index $b_2$ and $b_1$. 

The base orbifold of the fibration induced on the quotient by  the Hopf fibration is a 2-sphere with two possible singular points of index $b_1b_2$.  The homology class of the fiber in the tubular neighborhood of $\alpha_1$ is $b_2\mu+b_1\lambda$ where $\mu$ is a meridian and $\lambda$ a longitude; for $\alpha_2$ the roles of $b_1$ and $b_2$ are inverted.
\end{Claim}

\begin{proof}
Generalize the argument used in  Case 3; $\phi^{\frac{mnr}{ab_1}}$ plays the role of $\rho_2$ and $\phi^{\frac{mnr}{ab_2}}$ plays the role of $\rho_1$. Note that the two generators now have orders which are coprime (whereas the orders were equal in Case 3).
\end{proof}

Now consider the group $G_1=G/\langle \phi^\frac{mnr}{ab_1},\phi^\frac{mnr}{ab_2}\rangle $ which acts  on $S^3/\langle \phi^\frac{mnr}{ab_1},\phi^\frac{mnr}{ab_2}\rangle $ respecting the fibration. The quotient $(S^3/\langle \phi^\frac{mnr}{ab_1},\phi^\frac{mnr}{ab_2}\rangle )/G_1$ with the fibration induced by the quotient is equivalent to $S^3/G$ with the fibration induced by the Hopf fibration of $S^3$. The group $G_1$ is generated by $\overline \phi$ and $\overline{\gamma}$, the projections of $\phi$ and $\gamma$; their action on the 3-sphere (the underlying topological space of $S^3/\langle \phi^\frac{mnr}{ab_1},\gamma^\frac{mnr}{ab_2}\rangle $) is the following:

\medskip

$$
\begin{array}{rl}
\overline \phi (z_1+z_2j)=& e^{\left(2i\pi b_2\frac{n-sm}{mnr}\right)}z_1+e^{\left(2i\pi b_1\frac{n+sm}{mnr}\right)}z_2j\\
\overline \gamma (z_1+z_2j)= &e^{\left(-2i\pi b_2\frac{1}{n}\right)}z_1+e^{\left(2i\pi b_1\frac{1}{n}\right)}z_2j
\end{array}
$$
\medskip

We consider the following factorization of $n$:

$$n=\prod_{p_k|a} p_k^{u_k} \cdot \prod_{q_l\nmid a} q_l^{v_l} $$

where the $p_k$'s and $q_l$'s are distinct (odd) prime factors of $n$. We define $\nu=2\left(\prod_{p_k|a} p_k^{u_k}\right)/a$. It follows that $(2n)/(a\nu)=\prod_{q_l\nmid a} q_l^{v_l}$ and $a/2$ are coprime.

\begin{center}
\begin{table}
\begin{tabular}{|ll|}
\hline
\multicolumn{2}{|c|}{} \\
\multicolumn{2}{|c|}{$G=\Phi((C_{mr}/C_{m},C_{nr}/C_{n})_s)$ (Family $1'$)}\\
\multicolumn{2}{|c|}{and}\\
\multicolumn{2}{|c|}{ $G=\Phi((D^*_{2mr}/C_{m},D^*_{2nr}/C_{n})_s)$ (Family $11'$)} \\
\multicolumn{2}{|c|}{} \\
\hline
&\\
We define:  & $h=\gcd(m,n)$ \\
						&  $m'=\frac{m}{h}$ \\
					 &  $n'=\frac{n}{h}$ \\ 
					 & $a=\gcd(n'-sm',m'+sn',m'n'r)$ \\
 & $b_1=\gcd(\frac{n'-sm'}{a},\frac{m'n'r}{a})$ \\
 &  $b_2=\gcd(\frac{n'+sm'}{a},\frac{m'n'r}{a})$ \\   
 &   $\nu $ minimal positive integer s.t.  $\gcd(\frac{2n'}{a\nu },\frac{a}{2})=1$ \\
& $d=\frac{\nu^2 a(n'+sm')+2n'm'r}{2a\nu b_2}$ \\ 
 & $g=\frac{\nu^2 a(n'-sm')-2n'm'r}{2a\nu b_1}$  \\
 & $f=\frac{m'n'r}{2b_1b_2}$ \\
 & $\overline{g}$ s.t. $g\overline{g}\equiv 1 \, \mod f$ \\
& $\overline{c}$ s.t. $\left(\nu s+r\frac{2n'}{a\nu }\right)\overline{c}\equiv 1 \, \mod n'r$ \\
&\\
\hline
\multicolumn{2}{|c|}{} \\
\multicolumn{2}{|p{12cm}|} {The orbifold $S^3/\Phi((C_{mr}/C_{m},C_{nr}/C_{n})_s)$ fibers over $S^2\left(\frac{nr}{2},\frac{nr}{2}\right)$ with local invariants $\frac{d \overline{c}b_2 h}{\frac{nr}{2}}$ and $-\frac{g \overline{c}b_1h}{\frac{nr}{2}}$ and Euler number $-\frac{2m}{nr}$.}\\
&\\
\multicolumn{2}{|p{12cm}|} {The underlying topological space of $S^3/\Phi((C_{mr}/C_{m},C_{nr}/C_{n})_s)$ is the lens space  $ L(f,d\overline{g})$.}\\
&\\
\multicolumn{2}{|p{12cm}|} {The singular set  of $S^3/\Phi((C_{mr}/C_{m},C_{nr}/C_{n})_s)$ is a link with at most two components of singular index $b_2h$ and $b_1h$ (if the singular index is 1 the corresponding component consists of non-singular points).}\\
\multicolumn{2}{|c|}{} \\
\hline
\multicolumn{2}{|c|}{} \\
\multicolumn{2}{|p{12cm}|} {The orbifold $S^3/\Phi((D^*_{2mr}/C_{m},D^*_{2nr}/C_{n})_s)$ fibers over $D^2\left(;\frac{nr}{2},\frac{nr}{2}\right)$ with local invariants $\frac{d \overline{c}b_2 h}{\frac{nr}{2}}$ and  $-\frac{g \overline{c}b_1h}{\frac{nr}{2}}$ and Euler number $-\frac{m}{nr}$.}\\
&\\
\multicolumn{2}{|p{12cm}|} {The underlying topological space of $S^3/\Phi((D^*_{2mr}/C_{m},D^*_{2nr}/C_{n})_s)$ is the 3-sphere.}\\
\multicolumn{2}{|c|}{} \\
\hline
\end{tabular}
\bigskip
\caption{Families $1'$ and $11'$.}
\label{pluto}
\end{table}
\end{center}

\bigskip

\begin{Claim} 
 The group $G_1$ is generated by the element $\overline \phi^\nu \overline \gamma^{\frac{2n}{\nu a}}$. 
\end{Claim}

\begin{proof}

The order of $G_1$ is $(mnr)/(2b_1b_2)$. The order of $\overline \phi^\nu $ is $(mnr)/(a\nu b_1b_2)$ 
and the order of $\overline{\gamma}^{\frac{2n}{\nu a}}$ is $(a\nu )/2.$ Since $\gcd(a,m)=1$ and $\gcd(a/2,r)=1$, the orders of $\overline{\phi}^\nu $ and $\overline{\gamma}^{\frac{2n}{\nu a}}$ are coprime and their product has the same order as $G_1$.
\end{proof}

We define now:

$$ g=\frac{\nu ^2a(n-sm)-2mnr}{2a\nu  b_1},\, d=  \frac{\nu ^2a(n+sm)+2mnr}{2a\nu  b_2},\, f = \frac{mnr}{2b_1b_2}. $$

The map $\overline \phi^\nu  \overline \gamma^{\frac{2n}{\nu a}}$ sends $z_1+z_2j$ to $e^{2\pi\frac{g}{f}i}z_1+e^{2\pi\frac{d}{f}i}z_2j$.  Since this map acts freely and has order $f$, the integers $g$ and $f$ are coprime, so are $d$ and $f$. 
We denote by $\overline d$ and $\overline g$ two integers such that  $d\overline{d}\equiv 1\, \mod f$ and  $g\overline{g}\equiv 1\, \mod f.$

\begin{remark}\label{relation2} We note that $\gcd(f,db_2-gb_1)=m$. To get this equality one considers that $db_2-gb_1=m\left(\nu  s+r \frac{2n}{a\nu }\right)$ and proves that $nr$ and $\nu  s+r \frac{2n}{a\nu }$ are coprime, using the following outline: any prime factor of $nr=(r)(2n/a\nu)(a/2)(\nu)$ must be a factor of one of the parenthesized integers, and by definition, any prime factor of $\nu$ is an odd prime factor of $a$, hence, a factor of $a/2$, so it suffices to show that $\nu  s+r \frac{2n}{a\nu }$ has no prime factors in common with $r,$ $2n/a\nu,$ or $a/2.$

Therefore,  using  $\gcd(\overline{d},f)=1,$ $\gcd(\overline{g},f)=1$ and $\gcd(b_1,b_2)=1$, it follows that  $m=\gcd(f,b_2-g\overline{d}b_1)=\gcd(fb_1,b_2-g\overline{d}b_1)$ and $m=\gcd(f,d\overline{g}b_2-b_1)=\gcd(fb_2,d\overline{g}b_2-b_1)$.

\end{remark}

\begin{Claim} 
The  fibered orbifold $S^3/G$ has as underlying topological space a lens space $L(f,d\overline{g}).$  The local invariants of the two exceptional fibers are $\frac{d\overline{c}b_2}{\frac{nr}{2}}$ and $\frac{-g\overline{c}b_1}{\frac{nr}{2}}$ where $\overline{c}$ is the inverse of $\nu  s+r \frac{2n}{a\nu }\,\mod nr.$
\end{Claim}

\begin{proof}
The action of $G_1$ on $S^3/\langle \phi^\frac{mnr}{ab_1},\phi^\frac{mnr}{ab_2}\rangle $, whose underlying topological space is a 3-sphere, is explicit and   the underlying topological space of $S^3/G$ can be understood.

To compute the invariants we consider the two tori $T_1$ and $T_2$ decomposing  $S^3/\langle \phi^\frac{mnr}{ab_1},\phi^\frac{mnr}{ab_2}\rangle $ and apply  Lemma~\ref{quoziente-toro}.
For $T_1$ the fiber is homologous in the boundary to $b_2\mu+b_1\lambda$,  the quotient of $T_1$ by $G_1$ is a solid torus and the fiber induced by the quotient has a slope $(b_2-g\overline{d}b_1)/(fb_1)$. By Remark~\ref{relation2}, the slope can be written as $\left(\frac{b_2-g\overline{d}b_1}{m}\right)/\left(\frac{fb_1}{m}\right)$ where denominator and numerator are coprime. We denote by  $\overline{c}$ an inverse of  $(db_2-gb_1)/m$ $\mod nr$ (by Remark~\ref{relation2}   $(db_2-gb_1)/m$ and $nr$ are coprime). The integer $d \overline{c}$  is an  inverse $\mod(fb_1)/m$ of the numerator of the slope and gives the invariant.  Using the definition of  $f$, $g$ and $d$ we obtain the thesis; for $T_2,$  the roles of $d$ and $g$ are reversed.
 \end{proof}

\bigskip
\noindent
\textbf{\large{Case 5}}. $G=\Phi((C_{2mr}/C_{2m},C_{2nr}/C_{2n})_s)$ (Family $1$)
\medskip

Also in this case by the same proof of the previous one we can easily obtain the base orbifold and the Euler number.

\medskip

\textbf{Claim $\mathbf{1^{\prime}}$.}
\textit{The base orbifold of the  fibration of $S^3/G$ induced by the Hopf fibration of $S^3$  is  $S^2$ with two singular points of index $nr$. The Euler number of the  fibration of $S^3/G$ is $-2m/nr.$}

\medskip

To compute the local invariants we have to distinguish some subcases depending on the parity  of certain indices.
A summary of the situation is given in Table~\ref{pippo}. In all subcases the strategy is similar to that of Case 4, but Claim 2 has to be substantially modified under certain conditions.
In the following we describe which subcases we have to consider and how  Claim 2 has to be modified for the critical subcases; where the computation follows exactly the same strategy as  Case 4, we skip details and  report directly the final results in Table~\ref{pippo}.

By the same argument of Remark~\ref{coprime},  we can suppose in our computation that $m$ and $n$ are coprime. We can also suppose that $s$ is odd, in fact if $s$ is even, $r$ has to be odd and  $s$ can be  substituted with $r-s$ obtaining a conjugate  group by Proposition~\ref{classificationS3} (using conjugation by $\Phi(1,j)\in\SO4$).

We introduce now some notation: 

Let $\phi$ be the map sending $z_1+z_2j$ to $e^{\left(2i\pi\frac{n-sm}{2mnr}\right)}z_1+e^{\left(2i\pi\frac{n+sm}{2mnr}\right)}z_2j$ and $\gamma$ the map sending $z_1+z_2j$ to $e^{\left(-2i\pi\frac{1}{2n}\right)}z_1+e^{\left(2i\pi\frac{1}{2n}\right)}z_2j$. These maps generate $G$. 

We denote  $a=\gcd(n+sm,n-sm,2mnr)$,  $b_1=\gcd(\frac{n-sm}{a},\frac{2mnr}{a})$ and $b_2=\gcd(\frac{n+sm}{a},\frac{2mnr}{a})$.

The first difference between different subcases is pointed out by the following proposition.

\begin{Proposition}
If $m$ and $n$ are both odd then we have $a=2\gcd(n,s)$, otherwise $a=\gcd(n,s)$.
\end{Proposition}

\begin{proof}
It is evident  that $\gcd(n,s)$ divides $a$.
The integer $a$ divides $2n=(n+sm)+(n-sm)$.  Since we suppose  $m$ and $n$ coprime, we have that $\gcd(a,m)=1$. This implies (using $2sm=(n+sm)-(n-sm)$) that $a$ also  divides $2s$. 
Hence $a$ divides $2\gcd(n,s)$, regardless of the parities of $m$ and $n$. 

If  $m$ or $n$ is even, then  (since $s$ is assumed to be odd) $n+sm$ is odd, and hence $a$ is odd. It then follows that $a$ divides $\gcd(n,s)$, so we get $a=\gcd(n,s)$.

If both $m$ and $n$ are odd, $a$ has to be even and $2\gcd(n,s)$ divides $a$, so we conclude that $a=2\gcd(n,s)$.

\end{proof}

If  $m$ or $n$ is even, then $a=\gcd(n,s)$ and the computation of the invariants follows exactly the same strategy as the previous case; the results are reported in Table~\ref{pippo}.

If $m$ and $n$ are both odd the situation is more complicated since  the statement of the  analogue  of Claim~\ref{fixed-point-subgroup}  depends on the 
parity of the
 indices $r/b_i$. 

\medskip

\textbf{Claim $\mathbf{2^{\prime}}$.}
\textit{Suppose that $m$ and $n$ are odd, we define $f_i=\gcd(2,r/b_i)$; the subgroup of $G$ generated by the elements with non-empty fixed point set is generated by the maps $(z_1+z_2j)\rightarrow (e^{\frac{2i\pi}{f_2b_2}}z_1+z_2j)$ and  $(z_1+z_2j)\rightarrow (z_1+e^{\frac{2i\pi}{f_1b_1}}z_2j)$}

\begin{proof}
In any case the element $\phi^{\frac{2mnr}{ab_2}}$ acts trivially on $\alpha_1$ and acts as a rotation of order $b_2$ on $\alpha_2$.  
Suppose that   $\phi^t\gamma^u$ acts trivially on $\alpha_1$, if we denote by $k$ the $\gcd$ of $t$ and $(2mnr)/(ab_2)$, we obtain that $(2mnr)/(ab_2k)$ divides $2n.$ 

If $r/b_2$ is odd we obtain that $(rm)/b_2$ divides $k$ and analogously to  Claim~\ref{fixed-point-subgroup} of Case 4 we obtain that $\phi^{\frac{2mnr}{ab_2}}$ generates the cyclic group of elements fixing pointwise $\alpha_1$.

If $r/b_2$ is even we have that  $(rm)/2b_2$ divides $k$ and we obtain an element fixing pointwise $\alpha_1$ of order $2b_2.$ In this case to generate the cyclic group of elements fixing pointwise $\alpha_1$ we use  $\phi^{\frac{mr}{2b_2}}\gamma^{\frac{-n-sm}{2b_2}}$ (this is well-defined since $2$ divides $r/b_2$, so $r/2b_2$ is an integer).

For the elements fixing pointwise $\alpha_2$ the situation is symmetric.

\end{proof}

Now we can consider various cases: $r$  is odd (and both $r/b_1$ and $r/b_2$ are odd); $r$  is even and exactly one between $r/b_1$ and $r/b_2$ is odd; $r$ is even and both $r/b_1$ and $r/b_2$ are even. In each of these cases we can repeat the strategy used for Family $1'$ and we obtain the results given in Table~\ref{pippo}.
We remark that, since $b_1$ and $b_2$ are coprime, if $r$ is even at least one between  $r/b_1$ and $r/b_2$ is even.

\bigskip
\noindent
\textbf{\large{Case 6}}. $G=\Phi((D^*_{4mr}/C_{2m},D^*_{4nr}/C_{2n})_s)$ (Family $11$) and \\ $G=\Phi((D^*_{2mr}/C_{m},D^*_{2nr}/C_{n})_s)$ (Family $11'$)
\medskip

These groups are the semidirect product  of  the abelian groups in Families $1$ and $1'$ and the group generated by  the involution $\Phi(j,j)$, corresponding to the map $z_1+z_2 j\rightarrow \overline{z_1}+\overline{z_2} j$.  We denote by $A$ the abelian subgroup of index 2 corresponding to $\Phi((C_{2mr}/C_{2m},C_{2nr}/C_{2n})_s)$ or to  $\Phi((C_{mr}/C_{m},C_{nr}/C_{n})_s)$. The element $\Phi(j,j)$ acts by conjugation on $A$ inverting each element, so these groups are generalized dihedral. 

We consider the Hopf fibration on $S^3$: the action of $\Phi(j,j)$ leaves invariant the fibers corresponding to real numbers with respect to the map $(z_1+z_2 j)\rightarrow z_1/z_2$, and 
the involution $\Phi(j,j)$ acts on each of these fibers as a reflection (fixing exactly two points).  The   involution induced by $\Phi(j,j)$ on the base orbifold of the Hopf fibration is  a reflection along a great circle containing the  images of $\alpha_1$ and $\alpha_2$. To better understand  the situation we can consider the quotient of $S^3/G$ as the quotient of $S^3/A$ by the involution induced on $S^3/A$ by $\Phi(j,j)$. The base orbifold of $S^3/G$ is the quotient of the base orbifold of $S^3/A$ by a reflection along a great circle containing the two cone points. We obtain that the base orbifold of $S^3/G$ is a disc with two corner reflectors on the boundary and  the indices of the corner reflectors are the same of the singular points of the base orbifold of $S^3/A$ (see Tables~\ref{pluto} and~\ref{pippo}). By Theorem~\ref{eulero} the Euler number is half the Euler number of $S^3/A.$

The action of $G$ leaves invariant both $T_1$ and $T_2$, the two solid tori defined at the beginning of this section. The involution $\Phi(j,j)$ acts on $\alpha_1$ and $\alpha_2$ (the cores of  the two tori)  as a reflection. 

This implies that  $T_1/G$ and $T_2/G$ are two solid pillows and, since $S^3/G$ can be obtained by gluing the two solid pillows along their boundaries,  the underlying topological space of $S^3/G$ is $S^3.$

By definition, the local invariants of the corner points are the same as the two exceptional fibers of $S^3/A$. 

All these results are collected in Tables~\ref{pluto} and \ref{pippo}.

\begin{table}
\begin{tabular}{|l|l|}
\hline
\multicolumn{2}{|c|}{} \\
\multicolumn{2}{|c|}{$G=\Phi((C_{2mr}/C_{2m},C_{2nr}/C_{2n})_s)$ (Family $1$)}\\
\multicolumn{2}{|c|}{and}\\
\multicolumn{2}{|c|} {$G=\Phi((D^*_{4mr}/C_{2m},D^*_{4nr}/C_{2n})_s)$ (Family $11$)} \\
\multicolumn{2}{|c|}{} \\
\hline
\multicolumn{2}{|c|}{} \\
\multicolumn{2}{|l|}{We define: $h=\gcd(m,n),$ $m'=\frac{m}{h},$ $n'=\frac{n}{h},$}\\
\multicolumn{2}{|l|}{$a=\gcd(n'-sm',m'+sn',2m'n'r),$}\\
\multicolumn{2}{|l|}{  $b_1=\gcd(\frac{n'-sm'}{a},\frac{2m'n'r}{a}),$ $b_2=\gcd(\frac{n'+sm'}{a},\frac{2m'n'r}{a}).$} \\  
\multicolumn{2}{|c|}{} \\

\multicolumn{2}{|l|}{Remark: W.L.O.G. we assume $s$ odd}\\
\multicolumn{2}{|c|}{} \\
 \hline
&\\
if $n'm'$ is even,  we define: & in both cases we define:  \\
$\nu$ minimal positive integer s.t. $\gcd(\frac{n'}{a\nu },a)=1$ &  $d=\frac{\nu^2 a(n'+sm')+2n'm'r}{f_2a\nu b_2}$  \\ 
 $f_1=f_2=1$ & $g=\frac{\nu^2 a(n'-sm')-2n'm'r}{f_1a\nu b_1}$\\
& $f=\frac{2m'n'r}{f_1f_2b_1b_2}$ \\
if $n'm'$ is odd,  we define: &   $\overline{g}$ s.t. $g\overline{g}\equiv 1 \, \mod f$\\
$\nu$  minimal positive integer s.t.  $\gcd(\frac{2n'}{a\nu },\frac{a}{2})=1$  & $\overline{c}$  s.t.  \\
																							$f_i=\left\{\begin{array}{ll}
																							2 &  \text{ if } \frac{r}{b_i} \text{ is even }\\
																							1 &  \text{ if } \frac{r}{b_i} \text{ is odd }
																							\end{array}\right.$ &  $\left(\nu s+r\frac{2n'}{a\nu }\right)\overline{c}\equiv 1 \, \mod n'r$ \\
&  \\
\hline
\multicolumn{2}{|c|}{} \\
\multicolumn{2}{|l|} {The orbifold $S^3/\Phi((C_{2mr}/C_{2m},C_{2nr}/C_{2n})_s)$ fibers over $S^2(nr,nr)$}\\

\multicolumn{2}{|l|} { with local invariants $\frac{d \overline{c}f_2b_2h}{nr}$ and $-\frac{g \overline{c}f_1b_1h}{nr}$ and Euler number $-\frac{2m}{nr}$.}\\
\multicolumn{2}{|l|} {The underlying topological space of $S^3/\Phi((C_{2mr}/C_{2m},C_{2nr}/C_{2n})_s)$ } \\
\multicolumn{2}{|l|} {is the lens space  $ L(f,d\overline{g})$.}\\
\multicolumn{2}{|l|} {The singular set  of $S^3/\Phi((C_{2mr}/C_{2m},C_{2nr}/C_{2n})_s)$ is a link with  } \\ 
\multicolumn{2}{|l|} {at most two components of singular index $f_2b_2 h$ and $f_1b_1 h$}\\
\multicolumn{2}{|l|} {(if the singular index is 1 the corresponding component } \\
\multicolumn{2}{|l|} {consists of non-singular points).}\\
\multicolumn{2}{|c|}{} \\
\hline
\multicolumn{2}{|c|}{} \\
\multicolumn{2}{|l|} {The orbifold $S^3/\Phi((D^*_{4mr}/C_{2m},D^*_{4nr}/C_{2n})_s)$ fibers over $D^2(;nr,nr)$}\\ \multicolumn{2}{|l|} { with local invariants   $\frac{d \overline{c}f_2b_2h}{nr}$ and $-\frac{g \overline{c}f_1b_1h}{nr}$ and Euler number $-\frac{m}{nr}$.}\\
\multicolumn{2}{|l|} {The underlying topological space of $S^3/\Phi((D^*_{4mr}/C_{2m},D^*_{4nr}/C_{2n})_s)$ }\\
\multicolumn{2}{|l|} { is the 3-sphere.}\\
\multicolumn{2}{|c|}{} \\
\hline
\end{tabular}
\bigskip
\caption{Families $1$ and $11$.}
\label{pippo}
\end{table}

\section{The quotient of $S^3$ by the remaining groups}\label{remaining}

We will now consider the remaining groups  of Table~\ref{subgroup} that leave invariant the Hopf fibration. We will treat several examples which explain the general method to compute the classification data for the quotient  fibered orbifold (the fibration is that induced by the Hopf fibration). These results are collected in Table~\ref{topolino}.  We also considered separately the families $G=\Phi(L/L_K,R/R_K)$ and $\overline{G}=\Phi(R/R_K,L/L_K)$, when they do not coincide. In the table, the group $\overline{G}$ appears with the same number as $G$, adding the suffix ``bis''. Due to the previous discussion, there is an orientation-reversing diffeomorphism form $S^3/G$ and $S^3/\overline{G}.$ 
If the action of $\overline{G}$ preserves the Hopf fibration, then the action of $G$ preserves the mirror image of the Hopf fibration, and the invariants for the corresponding Seifert fibration of $S^3/G$ are closely related to those which we describe in cases 7-10 below. We give an example in Subsection~\ref{last}. 
We note that $S^3/G$ (and $S^3/\overline{G}$) may have the structures of a Seifert fibered orbifold in other ways, but that these additional structures would lift to $S^3$ as (classical) Seifert fibrations, having torus knots as generic fibers, and hence projecting to a 2-orbifold of the form $S^2(p,q)$, where $p$ and $q$ are coprime (and $|pq|>1$). This implies that the base 2-orbifold for $S^3/G$ is ``bad" in the sense that its universal cover has non-empty singular set. We will not consider Seifert fibration over bad 2-orbifolds. 
\bigskip
\bigskip

\noindent
\textbf{\large{Case 7}}.  $L$ is cyclic   and $R$ is generalized quaternion.
\medskip

This is the case of Families 2,\,3,\,4 and 34 in the list. We will consider explicitly the examples of Families 2 and 3, namely $G=\Phi((C_{2m}/C_{2m},D_{4n}^*/D_{4n}^*))$ and $G=\Phi((C_{4m}/C_{2m},D_{4n}^*/C_{2n}))$.

Similarly to Case 4 in the previous Section, the first step is to compute the induced action on the base 2-sphere. Since the elements of $L$ are of the form $\cos(\pi/2m)+i \sin(\pi/2m)$ and their induced action on the base 2-sphere is the identity (see the final part of Section \ref{Seifert orbifold}), it suffices to look at the action of elements of $\left\{1\right\}\times R$. We already know that the subgroup $C_{2n}\subset D_{4n}^*$ induces a rotation of order $n$ around an axis (say, the vertical axis) of $S^2$. One can see that the induced actions of the remaining elements (those of the form $\omega j$ for $\omega=\cos(\pi/2n)+i \sin(\pi/2n)$) are maps $\lambda\mapsto-(1/\lambda)\omega^{-2}$, which correspond to $\pi$-rotations around $n$ distinct axes intersecting the equator of $S^2$. This shows that the base orbifold of the quotient $S^3/G$ is $S^2(2,2,n)$. This information enables also to compute the Euler number, by using the naturality property  (see Theorem~\ref{eulero}).

Therefore, it remains to compute the local invariants associated to exceptional fibers, that is, to those fibers with image a singular point of the base orbifold. To do so, we will choose a preimage $\alpha$ in $S^3$ of one such exceptional fiber of $S^3/G$. The local invariant will only depend on the subgroup of $G$ fixing $\alpha$. Of course, the result will not depend on the chosen preimage, as the stabilizers of fibers of $S^3$ which are mapped to the same fiber of $S^3/G$ are conjugate.

In the two particular cases we are considering, there are three exceptional fibers. One projects to the cone point of index $n$; its preimages are the cores $\alpha_1$ and $\alpha_2$ of the solid tori $T_1$ and $T_2$ of $S^3$ in the usual decomposition. It is clear that the subgroup fixing them is $\Phi((C_{2m}/C_{2m},C_{2n}/C_{2n})_1)$ (the subindex $s=1$ will be omitted from now on). By using the results of the previous section, we can find the local invariant associated to the index $n$ cone point, which turns out to be $m/n$. The singularity index of this fiber is $\gcd(m,n)$.

The other two exceptional fibers project to index $2$ singular points. Take one fiber $\beta$ in the preimage of an exceptional fiber. 
Conjugate $G$ by an isometry of the form $\eta=\Phi(1,w_1+w_2j)$ which maps $\beta$ to $\alpha_1$, so that the invariants of $\beta$ equal the invariants of $\alpha_1$ under the action of $\eta G \eta ^{-1}.$
With this procedure, the stabilizer of $\beta$ can be conjugated to obtain the canonical form $(C_{\cdot\cdot}/C_{\cdot},C_{\cdot\cdot}/C_{\cdot})$, for which we know how to find the associated invariants thanks to the previous section.

For the groups in Family 2, that is $G=\Phi((C_{2m}/C_{2m},D_{4n}^*/D_{4n}^*))$, one can recognize that the subgroup leaving invariant any exceptional fiber over an index 2 singular point is conjugate to $\Phi((C_{2m}/C_{2m},C_4/C_4))$. For example, the elements of $R=D_{4n}^*$ fixing the fiber $z_1/z_2=i$ are $\left\{1,j,-1,-j\right\}$ and they are easily conjugated to  $C_4=\left\{1,i,-1,-i\right\}$. According to the results obtained for Family 1, the local invariant is $0/2$ when $m$ is even (this means that the fiber has singularity index 2 but has a trivially fibered neighborhood) and $1/2$ when $m$ is odd. On the other hand, for $G=\Phi((C_{4m}/C_{2m},D_{4n}^*/C_{2n}))$ the stabilizers are conjugate to  $\Phi((C_{4m}/C_{2m},C_4/C_2))$ and thus the local invariants are inverted: $1/2$ for $m$ even and $0/2$ for $m$ odd.

Note that every stabilizer acts on $S^3$ fixing two different fibers (they correspond to the two antipodal points fixed by the rotation on the base orbifold).  For Families 2 and 3, the local invariants associated to the two fibers are equal. This will not always be  the case. Families 4 and 34 are dealt  with the same techniques, by paying attention to the remarks above.

\bigskip

\noindent
\textbf{\large{Case 8}}.  $L$ and $R$ are generalized quaternion
\medskip

These are families of groups containing the groups of Case 7 as index 2 subgroups, listed as 10, 12, 13, 33, $33'$ in the table. 

Compared to Case 7, the additional elements are of the form $(\omega j,\omega' j)$, where $\omega$ and $\omega'$ are roots of unity. Again, we start by considering the induced action on $S^2$. One sees that the induced action for left multiplication by an element $(\omega j,1)$ is the antipodal map. Therefore, when a quaternion of the form $\omega j$ in $L$ is paired to a $\pi$-rotation arising from some $\omega' j$ in $R$, the induced action is reflection in the plane orthogonal to the axis of rotation. The reflection planes for the action of the group may or may not contain the axis of some other $\pi$-rotation. This will depend on the isomorphism
between $L/L_K$ and $R/R_K$.

For instance, we consider  the groups $\Phi((D^*_{4m}/D^*_{4m},D^*_{4n}/D^*_{4n}))$ in Family 10. If $n$ is odd, reflection planes do not contain any axis of $\pi$-rotations and therefore the quotient orbifold is $D^2(2;n)$; if $n$ is even, the quotient is $D^2(;2,2,n)$.

In order to compute local invariants, the procedure is the same as the previous case. Note that when an exceptional fiber projects to a corner reflector of the base orbifold, its local model is a solid pillow, and the stabilizer of one of its preimages in $S^3$ is a dihedral group. However, it suffices to detect the index 2 cyclic subgroup of this dihedral group to obtain the local invariant. Therefore, one can forget about the elements which act on the 2-sphere by reflections: these are exactly those arising from the pairing of some $\omega j$ in $L$ (it induces an antipodal map) to some $\omega' j$ in $R$ (induces a $\pi$-rotation). This shows that the local invariants will match those we obtained for the respective groups of Case~7.

\bigskip

\noindent
\textbf{\large{Case 9}}.  $L$ is generalized quaternion  and $R$ is cyclic 
\medskip

This case covers Families 2bis, 3bis, 4bis, 34bis. The technique is very similar to Case 8, though simpler. Let us look at the induced action on $S^2$. Elements of $R$ act by rotations, while elements of $L$  act either trivially or by the antipodal map. Depending on the pairing, the base orbifold can be some $D^2(n;)$ or $\mathbb{R}P^2(n)$. For example, consider the family 2bis, that is $\Phi((D^*_{4m}/D^*_{4m},C_{2n}/C_{2n}))$. If the order $n$ of the induced rotation of elements of $R$ is even, there is a $\pi$-rotation paired to an antipodal map, giving rise to a reflection in the horizontal plane; otherwise, there is no reflection in the induced action, so the only orientation-reversing maps act freely and the base orbifold is a projective plane with a cone point of order $n$. The local invariant of the only exceptional fiber is again $m/n$, as the stabilizer is $\Phi((C_{2m}/C_{2m},C_{2n}/C_{2n}))$.

\bigskip

In  cases 7-9, the orbifold $S^3/G$ has two different fibrations, one induced from the Hopf fibration and the other from the mirror image of the Hopf fibration. The latter can be recovered by looking at the quotient of $S^3$ (with the Hopf fibration) for the action of $\overline{G}$, and just changing the sign of  Euler number and local invariants due to orientation.

\bigskip
\bigskip

\noindent
\textbf{\large{Case 10}}.  $L$ is cyclic or generalized quaternion and  $R=T^*$, $O^*$, $I^*$
\medskip

This is the case of the remaining groups preserving the Hopf fibration. Note that these groups do not  preserve the mirror image of the Hopf fibration.

The groups of symmetries $T^*$, $O^*$ and $I^*$ act on $S^2$ as one should expect. In particular, $T^*$ has a normal subgroup $D_8^*$ whose action is a special case of $D_{4m}^*$ considered above. Moreover, $T^*$ has threefold axes of rotation, and  can be regarded as the normal subgroup of $O^*$ which leaves invariant a tetrahedron embedded in a cube as in Figure~\ref{cubo-tetraedro}. $O^*$ contains also fourfold axes. $I^*$ has twofold, threefold and fivefold axes.

\begin{figure}[htb]
\begin{center}
\includegraphics[height=3cm]{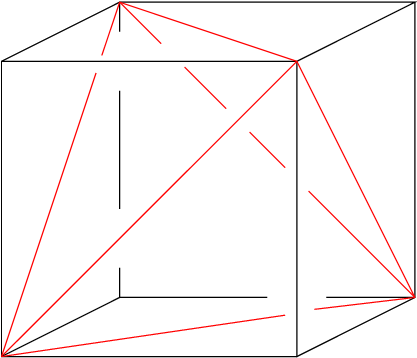}
\end{center}
\caption{A tetrahedron inside a cube}\label{cubo-tetraedro}
\end{figure}

When $L$ is cyclic, it has trivial induced action on the base $S^2$. Therefore, the base orbifold of the quotient is $S^2(2,3,3)$, $S^2(2,3,4)$ and $S^2(2,3,5)$ respectively. It is now easy to find the local stabilizers: take, for example, the groups $\Phi((C_{2m}/C_{2m},O^*/O^*))$ in Family 7. An exceptional fiber corresponding to the index 2 singular point is fixed by a subgroup conjugate to $\Phi((C_{2m}/C_{2m},C_4/C_4))$; a fiber corresponding to index 3 cone point is fixed by a group of the form $\Phi((C_{2m}/C_{2m},C_6/C_6))$ and analogously for the index 4 cone point we get a group  $\Phi((C_{2m}/C_{2m},C_8/C_8))$. On the other hand, for the groups $\Phi((C_{4m}/C_{2m},O^*/T^*))$ in Family 8, stabilizers are respectively $\Phi((C_{4m}/C_{2m},C_4/C_2))$ (the twofold axes of symmetry through edge midpoints of the cube are not a symmetry of the embedded tetrahedron), $\Phi((C_{2m}/C_{2m},C_6/C_6))$ (as the threefold axis is common for the groups $O^*$ and $T^*$) and $\Phi((C_{4m}/C_{2m},C_8/C_4))$ (of the order 4 cyclic subgroup of symmetries of a cube, only the order 2 rotation preserves the tetrahedron). This enables us  to find  the local invariants, which we report in Table \ref{topolino}. When $L=D_{4m}^*$, we have some antipodal maps paired to the order 2 rotations in $O^*$, thus obtaining $D^2(;2,3,4)$ as base 2-orbifold; anyway, the local invariants are obtained in the same way. The case $R=I^*$ is completely analogous.

When $R=T^*$ there are some more cases to be considered,  the base orbifold of the quotient is either  $D^2(;2,3,3)$ or $D^2(3;2)$ depending on the pairing of the elements of $R=T^*$ and of $L=D_{4m}^{*}$ (the 2-orbifolds $D^2(;2,3,3)$ or $D^2(3;2)$ are the two possible quotients of $S^2(2,3,3)$ by a reflection). In any case the method to compute the invariants is similar to the previous family.

\begin{center}
\begin{table}
\begin{tabular}{|l|c|c|c|c|c|}
\hline
  & group & e & base orbifold & invariants & case \\
\hline
 2. & $(C_{2m}/C_{2m},D^*_{4n}/D^*_{4n})$ & $-\frac{m}{n}$ & $S^2(2,2,n)$ & $\frac{m}{n},\frac{m}{2},\frac{m}{2}$ & \\ 
 3. &  $(C_{4m}/C_{2m},D^*_{4n}/C_{2n})$ & $-\frac{m}{n}$ & $S^2(2,2,n)$ & $\frac{m}{n},\frac{m+1}{2},\frac{m+1}{2}$ & \\ 
 4. &  $(C_{4m}/C_{2m},D^*_{8n}/D^*_{4n})$ & $-\frac{m}{2n}$ & $S^2(2,2,2n)$ & $\frac{m+n}{2n},\frac{m}{2},\frac{m+1}{2}$ & \\ 
34. &   $(C_{4m}/C_{m},D^*_{4n}/C_{n})$ & $-\frac{m}{2n}$ & $S^2(2,2,n)$ & $\frac{[(m+n)/2]}{n},\frac{m}{2},\frac{m+1}{2}$ & $m,n$ odd\\
10. &  $(D^*_{4m}/D^*_{4m},D^*_{4n}/D^*_{4n})$ & $-\frac{m}{2n}$ & $D^2(;2,2,n)$ & $\frac{m}{n},\frac{m}{2},\frac{m}{2}$ & $n$ even  \\
 & & & $D^2(2;n)$ & $\frac{m}{n},\frac{m}{2}$ & $n$ odd \\
13.bis &  $(D^*_{4m}/C_{2m},D^*_{8n}/D^*_{4n})$ & $-\frac{m}{2n}$ & $D^2(;2,2,n)$ & $\frac{m}{n},\frac{m}{2},\frac{m}{2}$ & $n$ odd  \\
 & & & $D^2(2;n)$ & $\frac{m}{n},\frac{m}{2}$ & $n$ even \\
 13. &   $(D^*_{8m}/D^*_{4m},D^*_{4n}/C_{2n})$ & $-\frac{m}{2n}$ & $D^2(;2,2,n)$ & $\frac{m}{n},\frac{m+1}{2},\frac{m+1}{2}$ & $n$ even  \\
 & & & $D^2(2;n)$ & $\frac{m}{n},\frac{m+1}{2}$ & $n$ odd \\
33. &  $(D^*_{8m}/C_{2m},D^*_{8n}/C_{2n})_f$ & $-\frac{m}{2n}$ & $D^2(;2,2,n)$ & $\frac{m}{n},\frac{m+1}{2},\frac{m+1}{2}$ & $n$ odd  \\
&  & & $D^2(2;n)$ & $\frac{m}{n},\frac{m+1}{2}$ & $n$ even \\
12. &   $(D^*_{8m}/D^*_{4m},D^*_{8n}/D^*_{4n})$ & $-\frac{m}{4n}$ & $D^2(;2,2,2n)$ & $\frac{m+n}{2n},\frac{m}{2},\frac{m+1}{2}$ & \\ 
$33^{\prime}$. &   $(D^*_{8m}/C_{m},D^*_{8n}/C_{n})_f$ & $-\frac{m}{4n}$ & $D^2(;2,2,n)$ & $\frac{[(m+n)/2]}{n},\frac{m}{2},\frac{m+1}{2}$ & $m,n$ odd \\

 2.bis &  $(D^*_{4m}/D^*_{4m},C_{2n}/C_{2n})$ & $-\frac{m}{n}$ & $D^2(n;)$ & $\frac{m}{n}$ & $n$ even \\
 &  & & $\mathbb{R}P^2(n)$ & $\frac{m}{n}$ & $n$ odd \\
3.bis &  $(D^*_{4m}/C_{2m},C_{4n}/C_{2n})$ & $-\frac{m}{n}$ & $D^2(n;)$ & $\frac{m}{n}$ & $n$ odd \\
 &  & & $\mathbb{R}P^2(n)$ & $\frac{m}{n}$ & $n$ even \\
4.bis & $(D^*_{8m}/D^*_{4m},C_{4n}/C_{2n})$ & $-\frac{m}{2n}$ & $D^2(2n;)$ & $\frac{m+n}{2n}$ & \\
34.bis &  $(D^*_{4m}/C_{m},C_{4n}/C_{n})$ & $-\frac{m}{2n}$ & $D^2(n;)$ & $\frac{[(m+n)/2]}{n}$ & $m,n$ odd \\

 5. & $(C_{2m}/C_{2m},T^*/T^*)$ & $-\frac{m}{6}$ & $S^2(2,3,3)$ & $\frac{m}{2},\frac{m}{3},\frac{m}{3}$ &  \\ 
 6. & $(C_{6m}/C_{2m},T^*/D^*_{8})$ & $-\frac{m}{6}$ & $S^2(2,3,3)$ & $\frac{m}{2},\frac{m+1}{3},\frac{m+2}{3}$ &  \\ 
16.  & $(D^*_{4m}/C_{2m},O^*/T^*)$ & $-\frac{m}{12}$ & $D^2(;2,3,3)$ & $\frac{m}{2},\frac{m}{3},\frac{m}{3}$ &  \\
18. & $(D^*_{12m}/C_{2m},O^*/D^*_{8})$ &  $-\frac{m}{12}$ & $D^2(;2,3,3)$ & $\frac{m}{2},\frac{m+1}{3},\frac{m+2}{3}$ &  \\ 
14.  & $(D^*_{4m}/D^*_{4m},T^*/T^*)$ & $-\frac{m}{12}$ & $D^2(3;2)$ & $\frac{m}{2},\frac{m}{3}$ & \\
 7. & $(C_{2m}/C_{2m},O^*/O^*)$ & $-\frac{m}{12}$ & $S^2(2,3,4)$ & $\frac{m}{2},\frac{m}{3},\frac{m}{4}$ &  \\ 
 8. &  $(C_{4m}/C_{2m},O^*/T^*)$ & $-\frac{m}{12}$ & $S^2(2,3,4)$ & $\frac{m+1}{2},\frac{m}{3},\frac{m+2}{4}$ &  \\ 
15. & $(D^*_{4m}/D^*_{4m},O^*/O^*)$ & $-\frac{m}{24}$ & $D^2(;2,3,4)$ & $\frac{m}{2},\frac{m}{3},\frac{m}{4}$ &  \\ 
17.  &  $(D^*_{8m}/D^*_{4m},O^*/T^*)$ & $-\frac{m}{24}$ & $D^2(;2,3,4)$ & $\frac{m+1}{2},\frac{m}{3},\frac{m+2}{4}$ &  \\ 
9. &  $(C_{2m}/C_{2m},I^*/I^*)$ & $-\frac{m}{30}$ & $S^2(2,3,5)$ & $\frac{m}{2},\frac{m}{3},\frac{m}{5}$ &  \\ 
19. &  $(D^*_{4m}/D^*_{4m},I^*/I^*)$ & $-\frac{m}{60}$ & $D^2(;2,3,5)$ & $\frac{m}{2},\frac{m}{3},\frac{m}{5}$ &  \\ 

\hline
\end{tabular}
\bigskip
\caption{The remaining groups}
\label{topolino}
\end{table}
\end{center}

\subsection{Summary: how to treat any group} 
In summary, these are the steps to check the classification data for the quotient by the action of any group $G$:
\begin{itemize}
\item Compute the induced action on the base 2-sphere and obtain its quotient spherical 2-orbifold, which is the base orbifold of the quotient $S^3/G$. 
\item Consider the total number of elements of $G$ and the order of the induced action on $S^2$, which gives the generic number of fibers of $S^3$ which are identified in $S^3/G$. By applying the naturality property, find the Euler number of $S^3/G$.
\item Choose a preimage $\alpha$ in $S^3$ of every exceptional fiber of the quotient. By conjugating the group (or more simply by considering the structure of the stabilizer of the corresponding point in the base $S^2$) detect the canonical form of the stabilizing subgroup of $\alpha$. Use tables \ref{pluto} or  \ref{pippo} to find the local invariants.
\end{itemize}
Actually, the last step in some cases can be performed by computing whether the preimage $\alpha$ is fixed pointwise by $G$ or not. Indeed, if for example the projection of $\alpha$ corresponds to an index 2 cone point, its local invariant can only be $0/2$ (this happens if $\alpha$ is fixed pointwise) or $1/2$ (if not). However, this procedure does not work in general for  the families  in Case 10.

\subsection{Some interpretation of the results} \label{last}
From these results, following \cite{Dun1}, it is possible to understand the underlying manifold and the singular set  for the fibered orbifolds obtained as a quotient of a fixed group. 

We know that if in the base 2-orbifold there are no boundary components, then the underlying manifold is a Seifert fibered manifold; its invariants can be easily deduced. Indeed the  classical invariant defined by Seifert in \cite{SeT} corresponds to the opposite of 	 the sum of the normalized local invariants and the Euler number. It turns out that the underlying manifolds are spherical and they are well known (see~\cite{O}.) For example, if the underlying manifold has at most two exceptional fibers,  it is a lens space obtained as a gluing of two solid tori which are preimages of two discs in the base orbifold.  The singular set of the orbifold will consist of a finite union of fibers (possibly empty).

When the base orbifold has one boundary component (possibly with some corner reflectors), the underlying space is a lens space and the singular set can be described in terms of rational tangles (see Proposition 2.11 in \cite{Dun1}). However, in this case the fibration does not correspond to a Seifert fibration for manifolds. Indeed, there are one or two singular curves of index 2 which are not fibers. In particular if  there are no cone points the underlying manifold is $S^3.$ 

Let us show two typical examples of how to compute these spaces. 

\begin{Example}{$\Phi((C_{2m}/C_{2m},D^*_{4n}/D^*_{4n}))$}\label{primo esempio}

Consider those groups  $\Phi((C_{2m}/C_{2m},D^*_{4n}/D^*_{4n}))$ (in Family 2) for which $m$ even. We have obtained that the base orbifold is $S^2(2,2,n)$ with local invariants $m/2\equiv 0/2$ for both fibers over index 2 cone points and $m/n$ for the third exceptional fiber, which we will call $\alpha$. While looking at the underlying manifold, we can forget about the fibers with $0/2$ invariant, as they have a trivially fibered neighborhood. On the other hand, let $m/n=m'/n'$ with $m'$ and $n'$ coprime. The fiber $\alpha$ has a fibered solid torus neighborhood which projects to a disc on the base orbifold containing the index $n$ cone point. The complement of this disc is another disc, whose preimage is a trivially fibered solid torus.  The two tori are glued together in such a way that fibers coincide and the underlying topological manifold is a lens space.     By Proposition 2.12 in \cite{Dun1}  the underlying space is the lens space $L(m',-a)$ where $a$ is the inverse of $n'$ $\mod m'$. Note that the two index 2 singular fibers form a 2-component link with  linking number 0, but each index 2 singular fiber has nonzero linking number with the exceptional fiber, whose singularity index is $\gcd(m,n)$.
\end{Example}

\begin{Example}{$\Phi((D^*_{4n}/D^*_{4n},C_{2m}/C_{2m}))$}\label{secondo esempio}

This is Family 2bis with $m$ and $n$ swapped. The quotient orbifolds of this example  are  orientation-reversing diffeomorphic to those in  Example \ref{primo esempio} and the fibration induced by the Hopf fibration here corresponds  to the fibration induced by the mirror image  of the Hopf fibration in the previous example.  
In the quotient, if $m$ is even, we have base orbifold $D^2(m;)$, Euler number $-n/m$ and non-normalized local invariant $n/m$. Let $\beta$ be the fiber projecting to the unique cone point. Let $\mu$ and $\lambda$ a meridian and a longitude of the fibered neighborhood of $\beta$. A generic fiber in a neighborhood of $\beta$ represents a curve $a\mu+m'\lambda$, where $a$ is  defined as in the previous example. This fibered neighborhood projects to a disc, whose complement is a neighborhood of the boundary component. The preimage of the complement is a solid torus where generic fibers on the boundary represent meridians. Since the sum of the local invariants coincides exactly with the opposite of the Euler number (see Proposition~\ref{somma eulero invarianti}), the invariant associated to the boundary component is $\xi=0$ and thus there are two index 2 singular curves bounding an annulus fibered by intervals. Both curves have nonzero linking number with  the exceptional fiber $\beta$. As the underlying space is composed of two solid tori where in the gluing a meridian of the second torus is glued to a $a\mu+m'\lambda$ curve, it can be described as a lens space $L(m',a)$ (matching, up to orientation, the underlying space in the previous example).
\end{Example}

To obtain the fibration induced by the mirror image of the Hopf fibration in the spherical orbifolds of Family 2, it is enough to invert the sign of  Euler number and  local invariant of  the fibrations obtained in Example \ref{secondo esempio}.  
The two examples  show  two different fibrations for the  orbifolds of Family 2, only one of which comes from a Seifert fibration of the underlying manifold. The same phenomenon occurs for all groups in Families 2, 3, 4, 13 and 34.

\subsubsection*{Acknowledgments}
We thank the referee for carefully reading our manuscript and for many helpful comments that improved the presentation of our work.

\end{document}